\documentclass[11pt,a4paper,leqno]{amsart}
\usepackage[T1]{fontenc}
\usepackage[utf8]{inputenc}
\usepackage{lmodern}
\usepackage[a4paper,margin=28mm]{geometry}
\usepackage{amsmath,amssymb,amsthm,bm}
\usepackage{graphicx,booktabs,float}
\usepackage{caption,adjustbox}
\usepackage{tikz}
\usetikzlibrary{patterns}
\usepackage[numbers,sort&compress]{natbib}

\usepackage[expansion=false]{microtype}
\usepackage[hidelinks,unicode]{hyperref}
\hypersetup{
  pdftitle={Upper bounds for double Roman domination and [k]-Roman domination of cylindrical graphs C(m) x P(n)},
  pdfauthor={Simon Brezovnik and Janez Zerovnik}
}
\newtheorem{thm}{Theorem}
\newtheorem{cor}[thm]{Corollary}
\newtheorem{lem}[thm]{Lemma}

\theoremstyle{definition}
\newtheorem{definition}[thm]{Definition}
\newtheorem{remark}[thm]{Remark}

\title[{Double Roman and $[k]$-Roman domination of cylindrical graphs}]{Upper bounds for double Roman domination and $[k]$-Roman domination of cylindrical graphs $C_m\Box P_n$}
\author[S. Brezovnik]{Simon Brezovnik}
\address{Faculty of Mechanical Engineering, University of Ljubljana, Ljubljana, Slovenia; Institute of Mathematics, Physics and Mechanics, Ljubljana, Slovenia}
\email{simon.brezovnik@fs.uni-lj.si}
\author[J. \v{Z}erovnik]{Janez \v{Z}erovnik}
\address{Faculty of Mechanical Engineering, University of Ljubljana, Ljubljana, Slovenia; Rudolfovo -- Science and Technology Centre, Novo Mesto, Slovenia}
\email{janez.zerovnik@fs.uni-lj.si}
\thanks{Corresponding author: Simon Brezovnik.}
\date{}

\begin{document}
\raggedbottom
\setlength{\emergencystretch}{1em}
\begin{abstract}
Roman-type domination parameters form an important class of graph
invariants that model protection and resource allocation problems on
networks. Among them, $[k]$-Roman domination provides a unified framework
that generalizes Roman, double Roman, and higher-order variants.
In this paper we investigate the $[k]$-Roman domination number of
cylindrical grids $C_m\Box P_n$ and derive several new constructive upper
bounds.
Our approach combines three complementary techniques: linear periodic
constructions, uniform ceiling-type labelings, and packing-based
refinements. We first analyze the case $C_9\Box P_n$, where these three
families of bounds can be compared explicitly and their relative
efficiency is shown to depend on the parameter $k$. We then extend the
linear constructions to cylindrical grids whose circumference is a
multiple of one of the values
{$ r \in \{3,4,5, \dots, 9\}$,}
 obtaining a unified family of
upper bounds for $C_{rt}\Box P_n$. Motivated by the asymptotic behavior of
these estimates, we further derive general upper bounds depending only on
the residue class of $m$ modulo $5$, which apply to all cylindrical grids.
As a consequence, we obtain explicit estimates for the double Roman
domination number $\gamma_{[2]R}(C_m\Box P_n)$ and compare the resulting
multiple-based constructions with the residue-class bounds. This
comparison shows that the residue-class construction becomes
asymptotically superior for all sufficiently large admissible
circumferences, while several exceptional small cases remain better
covered by tailored constructions.
\end{abstract}
\maketitle

\section{Introduction}

Domination-type parameters form one of the central topics of modern graph theory and arise naturally in applications such as facility location, monitoring of communication networks, and deployment of emergency or defense resources~\cite{Haynes,GareyJohnson}. {Such models are also relevant for large-scale networked systems,
where limited resources must be placed at selected nodes in order to
provide local protection, monitoring, or recovery capacity.
In communication, transportation, and infrastructure networks, failures
or attacks at one node may have to be compensated by nearby protected
nodes.
Although the present paper is theoretical, cylindrical grids provide a
regular product graph model in which one can study how local protection
rules give rise to global resource requirements.
This makes them a useful test class for understanding periodic
protection patterns in structured networks.
For broader perspectives on large-scale graph processing and modern
graph-based learning methods, see for example
\cite{Gan2025GraphCSR,Khemani2024GNNReview}.} Among these parameters, Roman domination has received considerable attention since its introduction by Cockayne et~al.~\cite{Cockayne2004}. A Roman dominating function assigns labels from $\{0,1,2\}$ to vertices so that each vertex labeled $0$ is protected by a neighboring vertex labeled $2$, and the minimum total weight of such a labeling defines the Roman domination number  $\gamma_R(G)$.

Several stronger variants of Roman domination have been introduced in order to model more robust protection mechanisms. These include double Roman domination
\cite{Beeler2016,D.Survey},
 triple Roman domination~\cite{ABDOLLAHZADEHAHANGAR2021125444}, and related higher-order extensions. A unified framework covering these variants was proposed by Abdollahzadeh Ahangar et~al.~\cite{ABDOLLAHZADEHAHANGAR2021125444} under the name $[k]$-Roman domination. In this setting, vertices receive labels from $\{0,1,\dots,k+1\}$ subject to local neighborhood constraints depending on both the assigned weight and the number of positively labeled neighbors. The corresponding minimum total weight defines the $[k]$-Roman domination number $\gamma_{[k]R}(G)$.

Roman-type domination parameters have been investigated on numerous graph classes, including trees, cycles, grids, and Cartesian product graphs (see, for example, refs.~\cite{Chellali2021VarietiesI,Chellali2020VarietiesII}).

Cartesian products of paths and cycles play a particularly important role in domination theory, since their regular structure supports periodic constructions while still exhibiting nontrivial combinatorial behavior~\cite{ImrichKlavzar,Chang1992,Alanko}. Among such graphs, cylindrical grids $C_m \Box P_n$ form a natural intermediate class between cycles, rectangular grids, and toroidal graphs.

{Recent work on \([k]\)-Roman domination has focused mainly on
general bounds, extremal graph classes, and complexity questions.
Valenzuela-Tripodoro et~al.~\cite{Ahangar2021} studied the
\([k]\)-multiple Roman domination framework as a common
generalization of double, triple, and quadruple Roman domination.
They derived sharp bounds and exact values for several graph
classes and considered the corresponding decision problem.
Khalili et~al.~\cite{Khalili2023} further developed the general
theory of \([k]\)-Roman domination by proving upper and lower
bounds for arbitrary graphs and trees, characterizing extremal
cases, and establishing NP-completeness for bipartite and chordal
graphs.

For double Roman domination, Anu and Lakshmanan~\cite{AnuAparna2024}
studied Cartesian products in detail. They improved a general
lower bound for products \(G\Box H\) when one factor has an
efficient dominating set, derived general upper bounds, and
obtained exact values for several products involving complete
graphs and cycles.

The present paper follows this direction, but works in the more
general \([k]\)-Roman domination framework and focuses specifically
on cylindrical grids \(C_m\Box P_n\). This setting permits more
refined constructions than those obtained from general graph
bounds alone. In particular, periodic labelings, ceiling-type
assignments, and packing-based reductions can be adapted to the
geometry of the cylinder and compared asymptotically as functions
of \(k\), \(m\), and \(n\).

More recently, \([k]\)-Roman domination on cylindrical grids has
been studied in~\cite{BrezovnikZerovnik2026_1,
BrezovnikZerovnikkRomanCylindrical2}. In
\cite{BrezovnikZerovnik2026_1}, the connection between
\([k]\)-Roman domination and efficient domination was used to
obtain explicit periodic \([k]\)-Roman dominating functions for
\(C_3\Box P_n\) and \(C_4\Box P_n\). Besides constructive upper bounds, lower bounds based on local
neighborhood constraints and efficient domination were also
established. In particular, for cylindrical grids one obtains a
general lower bound of the form
\[
\gamma_{[k]R}(C_m\Box P_n)>\frac{(k+1)mn}{5}.
\]
The efficient cases characterize when the minimum local neighborhood
weight can be attained at every vertex. These lower bounds make it possible to compare the
asymptotic behaviour of the constructed upper bounds with the best-known general lower estimates.

The analysis was extended in
\cite{BrezovnikZerovnikkRomanCylindrical2} to cylindrical grids
\(C_m\Box P_n\) for \(m\in\{5,6,7,8\}\). In particular, explicit
linear constructions, uniform ceiling-type bounds, and
packing-based refinements were developed and systematically
compared depending on the parameter \(k\) and the length of the
path. The obtained lower bounds furthermore showed that several
of these constructions are asymptotically close to optimal for
fixed small circumferences.}

In this paper, we continue this
{avenue}
 of research by extending the analysis to {a} general class of cylindrical graphs. Our approach combines three complementary techniques: periodic linear labeling along the cycle direction, uniform ceiling-type constructions based on neighborhood size, and packing-based refinements exploiting disjoint closed neighborhoods.

We first consider the case $C_9 \Box P_n$, where several competing
constructions yield different upper bounds depending on the parameter
$k$. A comparison of these bounds shows that linear, uniform, and
packing-based constructions become optimal in different parameter
regimes. We then extend the linear constructions obtaining general bounds for graphs
of the form $C_{rt} \Box P_n$, where
{$ r \in \{3,4, \dots, 9\}$. }
Motivated by the observation that decompositions into blocks of length
five yield the smallest asymptotic slopes with respect to $n$ for
$k\ge2$, we then derive a unified family of upper bounds depending on
the residue class of $m$ modulo $5$. This leads to explicit estimates
for $\gamma_{[k]R}(C_m\Box P_n)$ valid for arbitrary cylindrical grids.
In the special case $k=2$, these results yield new upper bounds for the
double Roman domination number, together with a comparison between the
multiple-based constructions and the general residue-class construction.

The paper is organized as follows. In Section~\ref{sec2} we introduce notation
and basic definitions. Section~\ref{Section3} presents three types of upper bounds
for $C_9\Box P_n$ and compares their efficiency for different values of
$k$. Section~\ref{sec4} extends the linear constructions to multiples of
{$ r \in \{3,4, \dots, 9\}$,} and develops general modulo-$5$ upper bounds for arbitrary
cylindrical grids. Section~\ref{sec5} specializes these results to the double
Roman domination number and compares the resulting estimates.
Concluding remarks and directions for future research are given in the
final section.

\section{Preliminaries}\label{sec2}

Throughout the paper we consider {finite simple graphs} $G=(V(G),E(G))$, where $V(G)$ denotes the vertex set and $E(G)$ denotes the edge set of $G$. For a vertex $v\in V(G)$, the \emph{open  neighborhood} of $v$ is defined as
$N(v)=\{u\in V(G)\mid uv\in E(G)\},$
while the \emph{closed neighborhood} of $v$ is
$N[v]=N(v)\cup\{v\}.$

For a subset $S\subseteq V(G)$ we write $N[S]=\bigcup_{v\in S}N[v].$

Let $f:V(G)\to\mathbb{N}_0$ be a function. For any subset $S\subseteq V(G)$ we define
$f(S)=\sum_{v\in S} f(v).$
The \emph{weight} of $f$ is given by
\[
\omega(f)=f(V(G))=\sum_{v\in V(G)}f(v).
\]

We next recall the definition of a \emph{$[k]$-Roman dominating function}~\cite{ABDOLLAHZADEHAHANGAR2021125444}. Let $k\in\mathbb{N}$. A function
\[
f:V(G)\to\{0,1,\dots,k+1\}
\]
is called a \emph{$[k]$-Roman dominating function} if for every vertex $v\in V(G)$ with $f(v)<k$ we have
\[
f(N[v])\ge k+|AN(v)|,
\]
where the
{ {\em active neighborhood}  $AN(v)$ of vertex $v$ is defined as}
\[
AN(v)=\{u\in N(v)\mid f(u)>0\}.
\]
The  minimum possible weight of such a function is called the \emph{$[k]$-Roman domination number} of $G$ and is denoted by
$\gamma_{[k]R}(G).$

{Intuitively, the condition
$f(N[v])\ge k+|AN(v)|$
can be interpreted in terms of military defense.
Vertices with positive labels represent defended cities containing
military units, while the value assigned to a vertex corresponds
to the defensive strength stationed at that location.
A vertex $v$ with $f(v)<k$ is not sufficiently protected on its
own and therefore requires reinforcement from its closed
neighborhood.

The term $|AN(v)|$ measures the number of neighboring cities that
already contain military units.
The additional term in the requirement accounts for the fact that
active neighboring cities may have to send units to support the
attacked city, while still maintaining sufficient protection in
the local neighborhood.

For example, let $k=5$ and suppose that a vertex $v$ has three
active neighbors.
Then the $[k]$-Roman domination condition requires
\[
f(N[v])\ge 5+3=8.
\]
Thus the cities in the closed neighborhood of $v$ collectively
contain at least eight military units: five units provide the
basic required level of defense, while the additional three units
correspond to the three active neighboring cities that may need
to contribute reinforcements without compromising local
protection.}

The parameter $\gamma_{[k]R}(G)$ generalizes several well-known domination-type invariants. In particular, for $k=1$ it coincides with the classical Roman domination number, while for $k=2$ it coincides with the double Roman domination number.

Since the graphs studied in this paper belong to the class of cylindrical grids, we briefly recall the definition of the Cartesian product. For graphs $G$ and $H$, \textit{Cartesian product} $G\Box H$ has vertex set $V(G)\times V(H)$, where vertices $(g_1,h_1)$ and $(g_2,h_2)$ are adjacent whenever either $g_1=g_2$ and $h_1h_2\in E(H)$, or $h_1=h_2$ and $g_1g_2\in E(G)$.

Let $C_m$ denote the cycle on $m$ vertices and $P_n$ the path on $n$ vertices. The Cartesian product $C_m\Box P_n$
is called a \emph{cylindrical grid}. Throughout the paper we interpret $C_m$ as the cyclic direction and $P_n$ as the longitudinal direction of the cylinder. Vertices of $C_m\Box P_n$ will be denoted by $(i,j)$, where
$i\in\{0,\dots,m-1\}, j\in\{0,\dots,n-1\}.$ {For each $j$, the subgraph induced by
$F_j=\{(i,j)\mid 0\le i\le m-1\}
$
is called the $j$-th fibre of $C_m \Box P_n$. We identify each fibre with its vertex set whenever convenient.}

Periodic constructions along the cyclic direction play a central role in the analysis of $[k]$-Roman domination on cylindrical grids. In particular, decompositions into blocks of length five will serve as the main building blocks in several later constructions.

We now describe an explicit periodic construction for the base case $m=5$, which was introduced in~\cite{BrezovnikZerovnikkRomanCylindrical2}. This construction will serve as a reference configuration in Section~\ref{sec5}, where it will be used to derive general upper bounds for $\gamma_{[k]R}(C_m\Box P_n)$ for arbitrary cylindrical grids and to obtain corresponding consequences for double Roman domination.

Consider the following periodic labeling pattern:

 {
\addtocounter{MaxMatrixCols}{2}

\begin{equation}
\begin{pmatrix}
  \cdots &  0 & \mid & k+1 & 0 & 0 & 0 & 0 & \mid & k+1 & \cdots \\
  \cdots &  0 & \mid & 0 & 0 & 0 & k+1 & 0 & \mid & 0 & \cdots \\
  \cdots &  0 & \mid & 0 & k+1 & 0 & 0 & 0 & \mid & 0 & \cdots \\
  \cdots &  k+1 & \mid & 0 & 0 & 0 & 0 & k+1 & \mid & 0 & \cdots \\
  \cdots &  0 & \mid & 0 & 0 & k+1 & 0 & 0 & \mid & 0 & \cdots
\end{pmatrix}.
\label{matrika5a}
\end{equation}
}

{The entry in row $i$ and column $j$ of
Equation~\eqref{matrika5a} represents the weight assigned to the
vertex $(i,j)$ of the cylindrical grid.
Thus entries equal to $0$ correspond to vertices without stationed
units, while entries equal to $k+1$ correspond to vertices carrying
the maximum defensive strength used in this pattern. The positive entries in Equation~\eqref{matrika5a} form a periodic
diagonal arrangement along the path direction.
From one fibre to the next, the position of the label $k+1$ is
shifted cyclically by two rows modulo $5$.
Consequently, every fibre in the displayed periodic block contains
exactly one vertex with positive weight, and the whole pattern
repeats with period five.

Since the pattern is periodic, it is enough to verify the
$[k]$-Roman domination condition on one block of five consecutive
fibres, together with the boundary corrections described below.
With this choice of the shifts, every vertex of weight $0$ in the
interior of the periodic pattern has exactly one vertex of weight
$k+1$ in its closed neighborhood.
Moreover, this vertex of weight $k+1$ is the only active vertex in
that closed neighborhood.
Hence, for every vertex $v$ with $f(v)=0$ whose closed neighborhood
is fully contained in the periodic part of the construction, we have
\[
f(N[v])=k+1
\qquad\text{and}\qquad
|AN(v)|=1.
\]
Therefore
\[
f(N[v])=k+1=k+|AN(v)|,
\]
so the \([k]\)-Roman domination condition is satisfied for all
interior vertices of weight \(0\).

Vertices in the first and last fibres, as well as vertices whose
closed neighborhoods meet these boundary fibres, have to be checked
separately, because one neighbouring fibre in the path direction is
missing. This is the reason for adding boundary corrections.
The following concrete example illustrates this correction for
\(C_5\Box P_8\), and it will serve as a reference
configuration for the modulo-\(5\) constructions used later.
It is obtained by taking eight consecutive fibres from the periodic
labeling in Equation~\eqref{matrika5a} and then adding two boundary
corrections, shown in bold, in the first and last fibres.
These bold entries, both equal to \(k\), compensate for the missing
neighbouring fibres beyond the two ends of the path and ensure that
the \([k]\)-Roman domination condition remains valid at the boundary.}

\begin{equation}
\begin{pmatrix}
0 & \mid & k+1 & 0 & 0 & 0 & 0 & \mid & k+1 & 0   \\
\boldsymbol{k} & \mid & 0 & 0 & 0 & k+1 & 0 & \mid & 0 & 0   \\
0 & \mid & 0 & k+1 & 0 & 0 & 0 & \mid & 0 &k+1    \\
k+1 & \mid & 0 & 0 & 0 & 0 & k+1 & \mid & 0 & 0  \\
0 & \mid & 0 & 0 & k+1 & 0 & 0 & \mid & 0 & \boldsymbol{k}
\end{pmatrix}.
\label{matrika5b}
\end{equation}

Finally, we recall a simple extension principle that allows periodic constructions on smaller cylindrical grids to be lifted to larger ones. Since cylindrical graphs are invariant under cyclic shifts along the $C_m$ direction, many constructions of $[k]$-Roman dominating functions can be obtained by repeating a fixed pattern.

More precisely, let $r$ be a divisor of $m$ and suppose that a $[k]$-Roman dominating function is defined on $C_r\Box P_n$. By repeating this labeling periodically along the cycle direction, we obtain a $[k]$-Roman dominating function on $C_m\Box P_n$, where $m=rt$ and $t\ge1$. The weight of the resulting labeling is $t$ times the weight of the original one. Consequently,
\[
\gamma_{[k]R}(C_{rt}\Box P_n)\le t\,\gamma_{[k]R}(C_r\Box P_n).
\]

This observation will be used repeatedly in the sequel to extend constructions from small base circumferences to their multiples.

{For later reference, we recall that explicit constructive upper
bounds for the base circumferences \(r\in\{3,4,\dots,8\}\) were
obtained in~\cite{BrezovnikZerovnik2026_1,
BrezovnikZerovnikkRomanCylindrical2}.
Together with the new construction for \(r=9\) developed in
Section~\ref{Section3}, these base estimates are the ingredients used
in Theorem~\ref{thm:linear_multiples_unified}.}

\section{Upper bounds for \texorpdfstring{$m=9$}{m=9}}
\label{Section3}

We now derive an upper bound for
$\gamma_{[k]R}(C_9 \Box P_n)$ by constructing
a suitable $[k]$-Roman dominating function.
{
Theorem~\ref{thm:m9} completes the analysis of the base cylindrical
grids \(C_m\Box P_n\) with \(3\le m\le 9\), extending previous work
that covered the cases \(m\in\{3,4\}\)
\cite{BrezovnikZerovnik2026_1} and \(m\in\{5,6,7,8\}\)
\cite{BrezovnikZerovnikkRomanCylindrical2}.
}
The resulting estimate is given below.
\begin{thm}\label{thm:m9}
Let $n\ge 2$. Then
\begin{equation}\label{eq:gamma_kR_C9_Pn}
\gamma_{[k]R}(C_{9}\Box P_n) \le 2n(k+1) + 2k.
\end{equation}
\end{thm}
\begin{proof}
We define a periodic labeling using the pattern
{\setcounter{MaxMatrixCols}{11}
\begin{equation}\label{pat:m3}
\begin{pmatrix}
k+1 &|&  0 & k+1  &0 &0 &\cdots &     k+1 & | & 0 &\cdots\\
0 &|& 0   &0 & 0 & 0 & \cdots &   0 & |&0 &\cdots \\
0 &|& k+1 &0 & k+1 &0 & \cdots &    0 & | &k+1 &\cdots \\
k &|& 0   &0 & 0 & 0 & \cdots &   0 & | &0 &\cdots \\
0 &|& 0   &k+1 &0 &k+1 & \cdots &    0 &  | &0 & \cdots \\
k+1 &|& 0 &0 & 0 & 0 & \cdots &    k+1 & | &0 & \cdots \\
0 &|& 0   &0 &k+1 &0 &  \cdots &    0 &  | &0 &\cdots \\
0 &|& k+1 &0 &0 &  0 &  \cdots & 0 & | &k+1 &  \cdots \\
0 &|& 0   &0 & 0 & k+1 &\cdots &    0 & | &0 & \cdots
\end{pmatrix}.
\end{equation}}

\noindent
Observe that the pattern requires a modification at the boundary.
Let $j\in\{0,1,\dots,8\}$ be such that $f(n-2,j)=f(n-2,j+4)=k+1$.
Then we set
\[
f(n-1,j-1)=k,\quad f(n-1,j+2)=k+1,
{ \text{~~and~~}
f(n-1,j+4)=f(n-1,j-3)=k+1,}
\]
where the indices are taken to be  modulo $9$, and all other vertices
{of the last fibre} $F_{n-1}$ are labeled $0$.

\noindent
The matrix, together with the adjustment of the last fibre, gives a $[k]$-RDF of the desired weight; therefore, the proof is complete.
\end{proof}

It turns out that for larger values of $k$, a more efficient strategy is to distribute the weight uniformly among all vertices. We will show in the subsequent analysis that such a uniform assignment yields a better upper bound compared to more irregular constructions. This observation motivates us to consider weight functions that assign equal contributions across the graph.

Based on this approach, we obtain the following result for the $[k]$-Roman domination number of the Cartesian product $C_9 \Box P_n$.

\begin{thm}\label{thm:C9Pn_uniform}
For $n\ge 4$,
\begin{eqnarray}
\gamma_{[k]R}(C_9 \Box P_n)
&\le&
9 (n-2)\left\lceil\frac{k+4}{5}\right\rceil
+
18\left\lceil
\frac{k+3-\left\lceil\frac{k+4}{5}\right\rceil}{3}
\right\rceil \label{eq:U_9}\\
&\le&
\frac{9nk+81n+6k-6}{5}.
\end{eqnarray}
\end{thm}

\noindent
\begin{proof}
For each interior fibre  {$F_i$, } $1 \le i \le n-2$, assign weight
\[
\left\lceil \frac{k+4}{5} \right\rceil
\]
to each vertex of the fibre.

On each boundary fibre (i.e., $i=0$ and $i=n-1$), assign weight
\[
\left\lceil
\frac{k+3-\left\lceil \frac{k+4}{5} \right\rceil}{3}
\right\rceil
\]
to each of its nine vertices.
Thus, each boundary fibre receives total weight
\[
9\left\lceil
\frac{k+3-\left\lceil \frac{k+4}{5} \right\rceil}{3}
\right\rceil.
\]

This assignment ensures that every vertex $v$ in an interior fibre satisfies
\[
f(N[v]) \ge 5\left\lceil \frac{k+4}{5} \right\rceil \ge k+4,
\]
since the closed neighborhood of an interior vertex consists of five vertices,
{implying $|AN(v)| \le 4$}.

On the other hand, every boundary vertex $v$ satisfies
\[
f(N[v]) \ge
\left\lceil \frac{k+4}{5} \right\rceil
+
3\left\lceil
\frac{k+3-\left\lceil \frac{k+4}{5} \right\rceil}{3}
\right\rceil
\ge k+3,
\]
because the closed neighborhood of a boundary vertex consists of its two neighbors in the same fibre, the vertex itself, and its unique neighbor in the adjacent interior fibre.

Therefore, $f$ is a $[k]$-RDF.
Its total weight is at most the contribution of the interior fibres,
which equals
\[
9(n-2)\left\lceil \frac{k+4}{5} \right\rceil,
\]
plus the contribution of the two boundary fibres, which equals
\[
18\left\lceil
\frac{k+3-\left\lceil \frac{k+4}{5} \right\rceil}{3}
\right\rceil.
\]
Hence,
\[
\gamma_{[k]R}(C_9 \Box P_n)
\le
9(n-2)\left\lceil \frac{k+4}{5} \right\rceil
+
18\left\lceil
\frac{k+3-\left\lceil \frac{k+4}{5} \right\rceil}{3}
\right\rceil.
\]

Using the inequalities $\lceil x \rceil \le x+1$ and
$\lceil x \rceil \ge x$, we obtain
\[
\left\lceil \frac{k+4}{5} \right\rceil \le \frac{k+4}{5}+1
= \frac{k+9}{5},
\]
and
\[
\left\lceil
\frac{k+3-\left\lceil \frac{k+4}{5} \right\rceil}{3}
\right\rceil
\le
\frac{k+3-\frac{k+4}{5}}{3}+1
=
\frac{4k+26}{15}.
\]
Substituting these bounds yields
\[
w(f)
\le
9(n-2)\frac{k+9}{5}
+
18\frac{4k+26}{15}
=
\frac{9nk+81n+6k-6}{5}.
\]
Therefore,
\[
\gamma_{[k]R}(C_9 \Box P_n)
\le
\frac{9nk+81n+6k-6}{5},
\]
which completes the proof.
\end{proof}

The construction described above is not tight, as the neighborhood sums frequently exceed the minimum required by the $[k]$-Roman domination condition. This excess allows for further optimization. In particular, for certain values of $k$, one can improve the bound by increasing the base weights to $\left\lceil \frac{k+5}{5} \right\rceil$ and then reducing the weight by one on vertices belonging to a suitably chosen packing set. {This motivates the
following notion.

\begin{definition}
A set $D \subseteq V(G)$ is called a \emph{packing} if
$N[u]\cap N[v]=\emptyset$ for every pair of distinct vertices
$u,v\in D$. The maximum cardinality of a packing in $G$ is called the
\emph{packing number} of $G$ and is denoted by $\rho(G)$.
\end{definition}

Packing sets will be used to decrease the weight on selected vertices while
preserving the Roman domination condition, which leads to improved upper bounds.}

\begin{lem}\label{lem:packing-C9Pn}
Let $n \geq 4$. The packing number of $C_9\Box P_n$ is
\[
\rho(C_9\Box P_n)= 2n-\left\lfloor \frac{n}{3}\right\rfloor.
\]
\end{lem}

{\begin{proof}
We first construct a packing of the required size. In fibres with two selected vertices, we choose vertices at cycle distance $4$
in $C_9$.
For every intermediate fibre we select one vertex,
shifted one position along the cycle, regarding to the previous intermediate fibre.
See Figure~\ref{fig:pattern-C9Pn}.

\noindent
{For optimality, let \(S\) be an arbitrary packing of
\(C_9\Box P_n\), and write
\[
s_i=|S\cap F_i|.
\]
The construction form Figure~\ref{fig:pattern-C9Pn} has average density \(5/3\), since its periodic
pattern selects \(2,2,1\) vertices in every three consecutive fibres.
We show that no packing can have larger~density.

Clearly, \(s_i\le 3\) for every fibre \(F_i\). Suppose first that
\(s_i=3\) for some \(i\). Then the three selected vertices in \(F_i\)
must be mutually at cyclic distance \(3\) in \(C_9\), say in rows
\(a,a+3,a+6\) modulo \(9\). Their closed neighborhoods cover the whole
fibre \(F_i\). Hence no vertex from either neighbouring fibre
\(F_{i-1}\) or \(F_{i+1}\) can belong to \(S\). Thus a fibre with three
selected vertices is isolated between two empty neighbouring fibres.
The local density of such a configuration is therefore at most
\(3/2\), which is smaller than \(5/3\). Consequently, fibres with
three selected vertices cannot improve on the density of the displayed
construction.

It remains to consider the case in which each fibre contains at most
two selected vertices. We claim that three consecutive fibres cannot
all contain two selected vertices. Suppose, to the contrary, that
\[
s_i=s_{i+1}=s_{i+2}=2.
\]
Let \(S\cap F_{i+1}=\{(a,i+1),(b,i+1)\}\). Since \(S\) is a packing,
the two vertices in \(F_{i+1}\) must be at a cyclic distance of at least
\(3\). Moreover, any vertex of \(S\cap F_i\) must avoid the rows
\(a\) and \(b\), because vertices in the same row in adjacent fibres
are adjacent. It must also avoid the rows adjacent to \(a\) and \(b\)
on the cycle, because otherwise the corresponding closed neighborhoods
would intersect at a vertex of \(F_{i+1}\). Thus the two vertices in
\(F_{i+1}\) forbid the six rows
\[
a-1,\ a,\ a+1,\ b-1,\ b,\ b+1
\]
of modulo \(9\)  in the neighbouring fibre \(F_i\). Only three rows remain
available. Among these three rows, two are necessarily at a cyclic
distance of at most \(2\), so two packing vertices cannot be placed in
\(F_i\), which constitutes  a contradiction.

Hence, in the case \(s_i\le2\) for all fibres, every three consecutive
fibres contribute at most five vertices. This gives density at most
\(5/3\), with at most two additional vertices in a remaining block of
one or two fibres. Therefore
\[
|S|\le
5\left\lfloor\frac n3\right\rfloor
+
2\left(n-3\left\lfloor\frac n3\right\rfloor\right)
=
2n-\left\lfloor\frac n3\right\rfloor.
\]
Together with the construction described above, this proves the lemma.}
\end{proof}}

\begin{figure}[htbp]
\begin{tikzpicture}[scale=0.75]

\foreach \i in {1,...,6}{
  \foreach \r in {0,1,2,3,4,5,6,7,8}{
    \draw (\i,-\r) rectangle (\i+1,-\r-1);
  }
}

\foreach \i in {0,...,5}{
  \node at (\i+1.5,0.6) {\small $F_{\i}$};
}

\foreach \r in {0,...,8}{
  \node[anchor=east] at (0.9,-\r-0.5) {\small $\r$};
}

\node at (1.5,-1.5) {\Large $\ast$};
\node at (1.5,-5.5) {\Large $\ast$};

\node at (2.5,-3.5) {\Large $\ast$};
\node at (2.5,-7.5) {\Large $\ast$};

\node at (3.5,-0.5) {\Large $\ast$};

\node at (4.5,-2.5) {\Large $\ast$};
\node at (4.5,-6.5) {\Large $\ast$};

\node at (5.5,-4.5) {\Large $\ast$};
\node at (5.5,-8.5) {\Large $\ast$};

\node at (6.5,-1.5) {\Large $\ast$};

\node at (8.5,-5) {\Large $\cdots$};

\end{tikzpicture}
\caption{Packing  pattern in $C_9\Box P_n$. The fibres follow a periodic pattern,
with a cyclic shift along $C_9$.}
\label{fig:pattern-C9Pn}
\end{figure}

{The role of Lemma~\ref{lem:packing-C9Pn} in the sequel is to make it
possible to reduce the weight of a suitable labeling on many vertices
simultaneously. The reason is that packing vertices have pairwise
disjoint closed neighborhoods, so a local reduction at the packing
vertices cannot accumulate in any single closed neighborhood.

\begin{remark}\label{rem:packing-reduction}
The packing-based reductions used below rely on the following simple
principle.
Suppose that a labeling \(f\) satisfies
\[
f(N[v])\ge k+|AN_f(v)|+1
\]
for every vertex \(v\), where \(AN_f(v)\) denotes the active
neighborhood with respect to \(f\).
If \(S\) is a packing, then the closed neighborhoods of distinct
vertices of \(S\) are pairwise disjoint.
Consequently, decreasing the weight of each vertex in \(S\) by one
affects every closed neighborhood by at most one unit.
Moreover, the reduction cannot increase the number of active neighbors.
Therefore, the reduction still guarantees the \([k]\)-Roman domination condition for the
modified labeling.
This is the mechanism used in the packing-based estimates below.
\end{remark}

We now state an upper bound for \(\gamma_{[k]R}(C_9\Box P_n)\) that
uses both the uniform ceiling-type labeling and the packing reduction
principle described above.}

\begin{thm}\label{thm:C9Pn_packing}
For $n\ge 4$,
\begin{align}
\gamma_{[k]R}(C_9\Box P_n)
\ \le\;&
9 (n-2) \left\lceil\frac{k+5}{5}\right\rceil
\ +\
18\left\lceil
\frac{k+{4}-\left\lceil\frac{k+5}{5}\right\rceil}{3}
\right\rceil -\left(2n-\left\lfloor \frac{n}{3}\right\rfloor\right) \label{eq:B_9}\\
\le\;&
{\frac{27nk+245n+18k}{15}}.
\end{align}
\end{thm}

\begin{proof}
For every interior fibre $1 \le i \le n-2$, assign to each of its vertices
the value
\[
\left\lceil \frac{k+5}{5} \right\rceil.
\]
For the boundary fibres, namely $i=0$ and $i=n-1$, assign the same value to
all nine vertices in the fibre, where each vertex receives weight
\[
\left\lceil
\frac{k+{4}-\left\lceil \frac{k+5}{5} \right\rceil}{3}
\right\rceil.
\]
This labeling is defined so that for every vertex $v$ we have
\[
f(N[v]) \ge k + |AN(v)| + 1.
\]

\noindent
{Let \(S\) be a maximum packing of \(C_9\Box P_n\).
By Lemma~\ref{lem:packing-C9Pn}, we have
\[
|S|=2n-\left\lfloor \frac{n}{3}\right\rfloor.
\]
As shown above, the initial ceiling-type labeling has one unit of
slack in every closed neighborhood relevant to the
\([k]\)-Roman domination condition.
Since \(S\) is a packing, Remark~\ref{rem:packing-reduction}
implies that the weight at each vertex of \(S\) can be reduced by
one without violating the \([k]\)-Roman domination condition.
Hence, the modified labeling remains a \([k]\)-RDF.}
Consequently,
\[
\gamma_{[k]R}(C_9\Box P_n)
\le
9 (n-2) \left\lceil\frac{k+5}{5}\right\rceil
+
18\left\lceil
\frac{k+{4}-\left\lceil \frac{k+5}{5}\right\rceil}{3}
\right\rceil
-
\left(2n-\left\lfloor \frac{n}{3}\right\rfloor\right),
\]
which establishes Inequality~\eqref{eq:B_9}.

\noindent
For the final estimate, we use the inequalities
\[
\left\lceil \frac{k+5}{5} \right\rceil \le \frac{k+5}{5}+1=\frac{k+10}{5}
\]
and
\[
\left\lceil
\frac{k+{4}-\left\lceil \frac{k+5}{5} \right\rceil}{3}
\right\rceil
\le
\frac{k+{4}-\frac{k+5}{5}}{3}+1
=
\frac{4k+{30}}{15}.
\]
Moreover,
\[
\left\lfloor \frac n3\right\rfloor \le \frac n3.
\]
Substituting these bounds yields
\[
\gamma_{[k]R}(C_9\Box P_n)
\le
9(n-2)\frac{k+10}{5}
+
18\frac{4k+{30}}{15}
-
2n+\frac n3
=
{\frac{27nk+245n+18k}{15}}.
\]
This completes the proof.
\end{proof}

To illustrate the transition from the linear bound to the uniform and packing-based constructions, we present in Figures~\ref{fig:best-bounds-C9Pn-k1-31-n4-20} and~\ref{fig:best-bounds-C9Pn-k50-60-n4-20} the optimal bounds for selected ranges of the parameter $k$.
The first figure focuses on small and intermediate values, namely
{ $k\in\{1,2, \dots,27\}$, }
where the linear bound dominates for very small $k$, while uniform Inequality \eqref{eq:U_9} and packing-based Inequality \eqref{eq:B_9} begin to appear in restricted regions as $k$~increases.

The second figure illustrates the situation for larger values of $k$, namely $k\in\{35,36, \dots,45\}$. In this regime, linear Inequality \eqref{eq:gamma_kR_C9_Pn} is no longer competitive, and the comparison is primarily between Inequalities \eqref{eq:U_9} and \eqref{eq:B_9}.
We observe that the optimal bound depends on both $k$ and $n$, with the uniform construction typically prevailing for larger values of $n$, while the packing-based construction may yield improvements for smaller values of $n$.
\begin{figure}[htbp]
\centering
\begin{tikzpicture}[
x=0.50cm,
y=0.40cm,
every node/.style={font=\small}
]

\def\nmin{4}
\def\nmax{20}
\def\kmin{1}
\def\kmax{27}

\def\gapshift{9}   \def\gapsize{0.4}

\foreach \n in {\nmin,...,\nmax} {

    \foreach \k in {1,2} {

        \pgfmathtruncatemacro{\vone}{2*\n*(\k+1) + 2*\k}

        \pgfmathtruncatemacro{\A}{ceil((\k+4)/5)}
    \pgfmathtruncatemacro{\B}{ceil((\k+3-\A)/3)}
    \pgfmathtruncatemacro{\vtwo}{9*(\n-2)*\A + 18*\B}

        \pgfmathtruncatemacro{\Athree}{ceil((\k+5)/5)}
    \pgfmathtruncatemacro{\Bthree}{ceil((\k+4-\Athree)/3)}
    \pgfmathtruncatemacro{\F}{floor(\n/3)}
    \pgfmathtruncatemacro{\vthree}{9*(\n-2)*\Athree + 18*\Bthree - (2*\n-\F)}

        \def\pat{north east lines}
    \pgfmathtruncatemacro{\best}{\vone}

    \ifnum\vtwo<\best
      \def\pat{vertical lines}
      \pgfmathtruncatemacro{\best}{\vtwo}
    \fi

    \ifnum\vthree<\best
      \def\pat{crosshatch}
      \pgfmathtruncatemacro{\best}{\vthree}
    \fi

    \pgfmathsetmacro{\x}{\n-\nmin}
    \pgfmathsetmacro{\y}{\k-\kmin}

    \fill[pattern=\pat] (\x,\y) rectangle ++(1,1);
    \draw[line width=0.2pt] (\x,\y) rectangle ++(1,1);
  }

    \foreach \k in {12,13,14,15,16,17,18,19,20,21,22,23,24,25,26,27} {

        \pgfmathtruncatemacro{\vone}{2*\n*(\k+1) + 2*\k}

        \pgfmathtruncatemacro{\A}{ceil((\k+4)/5)}
    \pgfmathtruncatemacro{\B}{ceil((\k+3-\A)/3)}
    \pgfmathtruncatemacro{\vtwo}{9*(\n-2)*\A + 18*\B}

        \pgfmathtruncatemacro{\Athree}{ceil((\k+5)/5)}
    \pgfmathtruncatemacro{\Bthree}{ceil((\k+4-\Athree)/3)}
    \pgfmathtruncatemacro{\F}{floor(\n/3)}
    \pgfmathtruncatemacro{\vthree}{9*(\n-2)*\Athree + 18*\Bthree - (2*\n-\F)}

        \def\pat{north east lines}
    \pgfmathtruncatemacro{\best}{\vone}

    \ifnum\vtwo<\best
      \def\pat{vertical lines}
      \pgfmathtruncatemacro{\best}{\vtwo}
    \fi

    \ifnum\vthree<\best
      \def\pat{crosshatch}
      \pgfmathtruncatemacro{\best}{\vthree}
    \fi

    \pgfmathsetmacro{\x}{\n-\nmin}
    \pgfmathsetmacro{\y}{\k-\kmin-\gapshift+\gapsize}

    \fill[pattern=\pat] (\x,\y) rectangle ++(1,1);
    \draw[line width=0.2pt] (\x,\y) rectangle ++(1,1);
  }
}

\draw[->] (0,0) -- ({(\nmax-\nmin+1)+0.8},0) node[below] {$n$};

\pgfmathsetmacro{\ymax}{(\kmax-\kmin-\gapshift+\gapsize)+1.8}
\draw[->] (0,0) -- (0,\ymax) node[left] {$k$};

\foreach \n in {\nmin,...,\nmax} {
  \pgfmathsetmacro{\x}{(\n-\nmin)+0.5}
  \draw (\x,0) -- (\x,-0.18);
  \node[below] at (\x,-0.18) {\scriptsize \n};
}

\foreach \k in {1,2} {
  \pgfmathsetmacro{\yt}{(\k-\kmin)+0.5}
  \draw (0,\yt) -- (-0.18,\yt);
  \node[left] at (-0.18,\yt) {\scriptsize \k};
}

\foreach \k in {12,13,14,15,16,17,18,19,20,21,22,23,24,25,26,27} {
  \pgfmathsetmacro{\yt}{(\k-\kmin-\gapshift+\gapsize)+0.5}
  \draw (0,\yt) -- (-0.18,\yt);
  \node[left] at (-0.18,\yt) {\scriptsize \k};
}
\end{tikzpicture}

\begin{tikzpicture}[x=0.8cm,y=0.5cm]
\fill[pattern=north east lines] (0,0) rectangle (1,1);
\draw (0,0) rectangle (1,1);
\node[right] at (1.2,0.5) {Inequality \eqref{eq:gamma_kR_C9_Pn}};

\fill[pattern=vertical lines] (6.0,0) rectangle (7.0,1);
\draw (6.0,0) rectangle (7.0,1);
\node[right] at (7.2,0.5) {Inequality \eqref{eq:U_9}};

\fill[pattern=crosshatch] (12.0,0) rectangle (13.0,1);
\draw (12.0,0) rectangle (13.0,1);
\node[right] at (13.2,0.5) {Inequality \eqref{eq:B_9}};
\end{tikzpicture}

\caption{The best bound among Inequalities \eqref{eq:gamma_kR_C9_Pn}, \eqref{eq:U_9}, and \eqref{eq:B_9} for $n=4,\dots,20$ and $k=1,\dots,27$ (with the gap between $k=2$ and $k=12$).}
\label{fig:best-bounds-C9Pn-k1-31-n4-20}
\end{figure}

\begin{figure}[htbp]
\centering
\begin{tikzpicture}[
x=0.50cm,
y=0.40cm,
every node/.style={font=\small}
]
\def\nmin{4}
\def\nmax{20}
\def\kmin{35}
\def\kmax{45}

\foreach \n in {\nmin,...,\nmax} {
  \foreach \k in {\kmin,...,\kmax} {

        \pgfmathtruncatemacro{\vone}{2*\n*(\k+1) + 2*\k}

        \pgfmathtruncatemacro{\A}{ceil((\k+4)/5)}
    \pgfmathtruncatemacro{\B}{ceil((\k+3-\A)/3)}
    \pgfmathtruncatemacro{\vtwo}{9*(\n-2)*\A + 18*\B}

        \pgfmathtruncatemacro{\Athree}{ceil((\k+5)/5)}
    \pgfmathtruncatemacro{\Bthree}{ceil((\k+4-\Athree)/3)}
    \pgfmathtruncatemacro{\F}{floor(\n/3)}
    \pgfmathtruncatemacro{\vthree}{9*(\n-2)*\Athree + 18*\Bthree - (2*\n-\F)}

        \def\pat{north east lines}     \pgfmathtruncatemacro{\best}{\vone}

    \ifnum\vtwo<\best
      \def\pat{vertical lines}       \pgfmathtruncatemacro{\best}{\vtwo}
    \fi

    \ifnum\vthree<\best
      \def\pat{crosshatch}       \pgfmathtruncatemacro{\best}{\vthree}
    \fi

        \pgfmathsetmacro{\x}{\n-\nmin}
    \pgfmathsetmacro{\y}{\k-\kmin}

        \fill[pattern=\pat] (\x,\y) rectangle ++(1,1);
    \draw[line width=0.2pt] (\x,\y) rectangle ++(1,1);
  }
}

\draw[->] (0,0) -- ({(\nmax-\nmin+1)+0.8},0) node[below] {$n$};
\draw[->] (0,0) -- (0,{(\kmax-\kmin+1)+0.8}) node[left] {$k$};

\foreach \n in {\nmin,...,\nmax} {
  \pgfmathsetmacro{\x}{(\n-\nmin)+0.5}
  \draw (\x,0) -- (\x,-0.18);
  \node[below] at (\x,-0.18) {\scriptsize \n};
}

\foreach \k in {\kmin,...,\kmax} {
  \pgfmathsetmacro{\y}{(\k-\kmin)+0.5}
  \draw (0,\y) -- (-0.18,\y);
  \node[left] at (-0.18,\y) {\scriptsize \k};
}
\end{tikzpicture}

\begin{tikzpicture}[x=0.8cm,y=0.5cm]
\fill[pattern=north east lines] (0,0) rectangle (1,1);
\draw (0,0) rectangle (1,1);
\node[right] at (1.2,0.5) {Inequality \eqref{eq:gamma_kR_C9_Pn}};

\fill[pattern=vertical lines] (6.0,0) rectangle (7.0,1);
\draw (6.0,0) rectangle (7.0,1);
\node[right] at (7.2,0.5) {Inequality \eqref{eq:U_9}};

\fill[pattern=crosshatch] (12.0,0) rectangle (13.0,1);
\draw (12.0,0) rectangle (13.0,1);
\node[right] at (13.2,0.5) {Inequality \eqref{eq:B_9}};
\end{tikzpicture}

\caption{The best bound among Inequalities \eqref{eq:gamma_kR_C9_Pn}, \eqref{eq:U_9}, and \eqref{eq:B_9} for $n=4,\dots,20$ and $k=35,\dots,45$ for $C_9\Box P_n$.}
\label{fig:best-bounds-C9Pn-k50-60-n4-20}
\end{figure}

These figures clearly demonstrate the transition thresholds between the different constructions and motivate the precise comparison given in Theorem~\ref{thm:C9Pn-U-B-beat-linear}.

\begin{thm}\label{thm:C9Pn-U-B-beat-linear}
For sufficiently large $k$, at least one of the Inequalities
\eqref{eq:U_9} or \eqref{eq:B_9} improves the linear Inequality
\eqref{eq:gamma_kR_C9_Pn}. More precisely:
\begin{itemize}
\item if $k\equiv 1\pmod 5$, then for $k\ge 26$ and all sufficiently large $n$,
      Inequality \eqref{eq:U_9} is the smallest among
      \eqref{eq:gamma_kR_C9_Pn}, \eqref{eq:U_9}, and \eqref{eq:B_9};
\item if $k\not\equiv 1\pmod 5$, then for $k\ge 38$ and all sufficiently large $n$,
      Inequality \eqref{eq:B_9} is the smallest among
      \eqref{eq:gamma_kR_C9_Pn}, \eqref{eq:U_9}, and \eqref{eq:B_9}.
\end{itemize}
\end{thm}

\begin{proof}
Fix $k$ and compare the three bounds as functions of $n$.

\noindent
Inequality \eqref{eq:gamma_kR_C9_Pn} has the form
\[
2n(k+1)+O_k(1).
\]

\noindent
Let
\[
a=\left\lceil \frac{k+4}{5}\right\rceil
\qquad\text{and}\qquad
b=\left\lceil \frac{k+3-a}{3}\right\rceil.
\]
Then Inequality \eqref{eq:U_9} equals
\[
9(n-2)a+18b
=
9an+O_k(1).
\]

\noindent
Let
\[
a'=\left\lceil \frac{k+5}{5}\right\rceil
\qquad\text{and}\qquad
b'=\left\lceil \frac{k+{4}-a'}{3}\right\rceil.
\]
Since
\[
2n-\left\lfloor \frac n3\right\rfloor
=
\frac53 n+O(1),
\]
Inequality \eqref{eq:B_9} equals
\[
9(n-2)a'+18b'-\left(2n-\left\lfloor \frac n3\right\rfloor\right)
=
\left(9a'-\frac53\right)n+O_k(1).
\]

\noindent
Therefore, for sufficiently large $n$, the smallest bound is determined
by comparing the linear coefficients
\[
2(k+1), \qquad 9a, \qquad 9a'-\frac53.
\]

\noindent
Comparing $9a$ with $2(k+1)$ shows that
\[
9a<2(k+1)
\]
holds when $k\equiv 1\pmod 5$ and $k\ge 26$.

\noindent
In the remaining residue classes, comparing $9a'-\frac53$ with $2(k+1)$ shows that
\[
9a'-\frac53<2(k+1)
\]
holds for all $k\ge 38$.

\noindent
Finally, if $k\equiv 1\pmod 5$, then $a<a'$, and hence
\[
9a<9a'-\frac53,
\]
so Inequality \eqref{eq:U_9} is asymptotically smaller than
Inequality \eqref{eq:B_9}.

If $k\not\equiv 1\pmod 5$, then $a=a'$, and therefore
\[
9a'-\frac53<9a,
\]
so Inequality \eqref{eq:B_9} is asymptotically smaller than
Inequality \eqref{eq:U_9}.

Thus, for sufficiently large $n$, Inequality \eqref{eq:U_9} is the smallest among
\eqref{eq:gamma_kR_C9_Pn}, \eqref{eq:U_9}, and \eqref{eq:B_9} whenever
$k\equiv 1\pmod 5$ and $k\ge 26$, while Inequality \eqref{eq:B_9} is the smallest among
these three bounds whenever $k\not\equiv 1\pmod 5$ and $k\ge 38$.
This completes the proof.
\end{proof}

\section{General upper bounds for \texorpdfstring{$\gamma_{[k]R}(C_m\Box P_n)$}{[k]-Roman domination}}\label{sec4}
The constructions derived  for the base cases $m=3,4,\ldots,9$
{  (see Section~\ref{Section3} and  ~\cite{BrezovnikZerovnik2026_1,BrezovnikZerovnikkRomanCylindrical2})
}
are local and periodic along the cycle direction.
In each case, the labeling is confined to a fixed number
of consecutive rows and depends only on neighborhood
conditions inside this block.
Consequently, these patterns can be replicated
independently along the cycle whenever the circumference
is a multiple of $3,4,\ldots,9$.
This observation allows us to extend all previously obtained
bounds to the graphs $C_{3t}\Box P_n$, $C_{4t}\Box P_n$,
$\ldots$, $C_{9t}\Box P_n$.

\subsection{Linear Constructions}
The periodic nature of the linear constructions allows them to be replicated along the cycle direction.
This observation leads to a unified formulation of the bounds for all circumferences that are multiples of $3,4,\ldots,9$.

\begin{thm}\label{thm:linear_multiples_unified}
Let $t\ge1$ and $n\ge2$.
For $m\in\{3t,\ldots,9t\}$ the following bounds hold:
\[
\gamma_{[k]R}(C_m\Box P_n)
\le
\begin{cases}
\displaystyle
t(nk+2),
& m=3t,\\[10pt]
\displaystyle
tn(k+1),
& m=4t,\\[10pt]
\displaystyle
t\bigl(n(k+1)+2k\bigr),
& m=5t,\\[10pt]
\displaystyle
t\Bigl(\left\lceil \frac{4n}{3}\right\rceil(k+1) + (k+1)\Bigr),
& m=6t,\ n\equiv 0 \pmod{3},\\[10pt]
\displaystyle
t\left\lceil \frac{4n}{3}\right\rceil(k+1),
& m=6t,\ n\equiv 1 \pmod{3},\\[10pt]
\displaystyle
t\Bigl(\left\lceil \frac{4n}{3}\right\rceil(k+1) + k\Bigr),
& m=6t,\ n\equiv 2 \pmod{3},\\[12pt]
\displaystyle
t\left((n+1)(k+1)+\frac{n-1}{2}(k+1)+2k\right),
& m=7t,\ n\text{ odd},\\[12pt]
\displaystyle
t\left(n(k+1)+\frac{n}{2}(k+1)+2k+1\right),
& m=7t,\ n\text{ even},\\[12pt]
\displaystyle
t\left(
2n(k+1)
-
\left(
\left\lfloor\frac{n-2}{5}\right\rfloor
+
\left\lfloor\frac{n}{5}\right\rfloor
\right)
+2k
\right),
& m=8t,\\[12pt]
\displaystyle
t\bigl(2n(k+1)+2k\bigr),
& m=9t.
\end{cases}
\]
\end{thm}

\begin{proof}
We partition the cycle of $C_m\Box P_n$ into $t$ consecutive blocks of size
$3$, $4$, $\ldots$, or $9$, according to the value of $m$.

\noindent
Let $m=rt$, where $r\in\{3,4,\ldots,9\}$.
Partition the cycle of $C_m\Box P_n$ into $t$ consecutive blocks, each inducing a copy of $C_r\Box P_n$.
On every such block we apply the corresponding construction for $C_r\Box P_n$.

\noindent
{For $r=9$, the construction is given in the present paper, while for
$r\in\{3,\dots,8\}$ we use the constructions established in our earlier work
on cylindrical grids $C_r\Box P_n$ (see~\cite{BrezovnikZerovnik2026_1,BrezovnikZerovnikkRomanCylindrical2}).
}

\noindent
Since each construction is defined locally within a block and depends only on adjacency relations inside that block, it remains valid when repeated periodically along the cycle.
Moreover, the blocks are disjoint in the cycle direction, so the total weight equals the sum of the contributions of all $t$ blocks.

\noindent
Consequently,
\[
\gamma_{[k]R}(C_m\Box P_n)
\le
t \cdot \gamma_{[k]R}(C_r\Box P_n),
\]
which yields the stated bounds for $m\in\{3t,\ldots,9t\}$ by substituting the corresponding estimates for $C_r\Box P_n$.
\end{proof}

As observed above, the same cylindrical graph $C_m\Box P_n$ may admit several distinct decompositions into blocks of admissible sizes $r\in\{3,4, \ldots,9\}$.
Each such decomposition yields a valid upper bound obtained by repeating the corresponding base construction.
Since these bounds are not necessarily equal, it is natural to compare them in order to determine which decomposition provides the best estimate.

{To compare these bounds, we first consider their asymptotic behavior as
functions of $n$ for fixed $k$.
More precisely, we determine the linear coefficient (slope) with respect
to $n$ in each case.
This comparison identifies which construction is asymptotically strongest
as \(n\to\infty\).

However, the bounds also contain additive terms depending on \(k\) and
on residue classes of \(n\).
These terms do not affect the asymptotic slope, but they may influence
which estimate is numerically best for moderate values of \(n\).
For this reason, the asymptotic comparison below is complemented in
Section~\ref{sec5} by an explicit finite comparison in the double Roman case.} If $m=rt$, then each bound is of the form $t\cdot B_r(n,k)$, and hence its slope equals $\frac{m}{r}$ times the slope of the corresponding base bound for $C_r\Box P_n$.

From the expressions in Theorem~\ref{thm:linear_multiples_unified}, we obtain the following slopes:
\[
\begin{aligned}
r=3&:\ \frac{mk}{3}, &
r=4&:\ \frac{m(k+1)}{4}, &
r=5&:\ \frac{m(k+1)}{5}, \\[2pt]
r=6&:\ \frac{2m(k+1)}{9}, &
r=7&:\ \frac{3m(k+1)}{14}, &
r=8&:\ \frac{m(5k+4)}{20}, \\[2pt]
r=9&:\ \frac{2m(k+1)}{9}.
\end{aligned}
\]

\noindent
{(Note that the slopes for $r=6$ and  $r=9$ are equal.)}
Since the factor $m$ is common to all expressions, it suffices to compare the normalized slopes
\[
\frac{k}{3},\quad
\frac{k+1}{4},\quad
\frac{k+1}{5},\quad
\frac{2(k+1)}{9},\quad
\frac{3(k+1)}{14},\quad
\frac{5k+4}{20}.
\]

A direct comparison shows that for all $k\ge 2$ we have
\[
\frac{k+1}{5}
<
\min\left\{
\frac{k}{3},\,
\frac{k+1}{4},\,
\frac{2(k+1)}{9},\,
\frac{3(k+1)}{14},\,
\frac{5k+4}{20}
\right\}.
\]
Hence, among all admissible constructions, the decomposition with $r=5$
(i.e., $m=5t$) yields the smallest slope for every $k\ge 2$.

For $k=1$, however, we have
\[
\frac{k}{3}=\frac{1}{3}
<
\frac{2}{5}=\frac{k+1}{5},
\]
so in this case the decomposition with $r=3$ gives the smallest slope.

Consequently, for sufficiently large $n$, the best linear bound is obtained
from the decomposition $m=5t$ whenever $k\ge 2$, and from $m=3t$ when $k=1$.

The case $k=1$ corresponds to classical Roman domination and has already been
extensively studied in the literature {\cite{Chang1992,Alanko,Kragujevac}}.
We therefore focus on the general case $k\ge 2$ and use the above observations
to derive upper bounds for arbitrary cylindrical graphs~$C_m\Box P_n$.

Motivated by the fact that the decomposition $m=5t$ yields the smallest slope
for all $k\ge 2$, we take the construction for $C_{5t}\Box P_n$ as a starting point
and extend it to general values of $m$.
More precisely, we modify the periodic pattern of width $5$ by locally removing
certain rows and adjusting the remaining labels so that the $[k]$-Roman domination
condition is preserved.
In this way, we obtain constructions for all residue classes of $m$ modulo $5$,
which allows us to derive upper bounds for all cylindrical graphs $C_m\Box P_n$. This leads to the following theorem.

\begin{thm}\label{thm:mod5-final}
Let $k\ge 2$ and $n\ge 4$. For $m\ge 3$,
\begingroup
\small
\[
\gamma_{[k]R}(C_m\Box P_n)
\le
\begin{cases}
\displaystyle
\frac{m}{5}\bigl(n(k+1)+2k\bigr),
& m\equiv 0 \pmod{5},\\[10pt]
\displaystyle
\frac{m+4}{5}\bigl(n(k+1)+2k\bigr)
-(4k+1)-2(k+2)\left\lfloor\frac{n-2}{5}\right\rfloor,
& m\equiv 1 \pmod{5},\\[10pt]
\displaystyle
\frac{m+3}{5}\bigl(n(k+1)+2k\bigr)
-3k-(k+3)\left\lfloor\frac{n-2}{5}\right\rfloor,
& m\equiv 2 \pmod{5},\\[10pt]
\displaystyle
\frac{m+2}{5}\bigl(n(k+1)+2k\bigr)
-2k-2\left\lfloor\frac{n-2}{5}\right\rfloor,
& m\equiv 3 \pmod{5},\\[10pt]
\displaystyle
\frac{m+1}{5}\bigl(n(k+1)+2k\bigr)
-2k,
& m\equiv 4 \pmod{5}.
\end{cases}
\]
\endgroup

\end{thm}

{\begin{proof}
For \(m\equiv 0 \pmod 5\), the claim follows directly from
Theorem~\ref{thm:linear_multiples_unified}.

\noindent
Assume now that \(m\not\equiv 0 \pmod 5\).
In each residue class, we obtain a labeling on \(C_m\Box P_n\) from the periodic
labeling on the next larger cylinder \(C_{m+r}\Box P_n\), where
\(r\in\{1,2,3,4\}\) is chosen so that \(m+r\equiv 0\pmod 5\).
We then delete \(r\) consecutive rows and modify the labels only in a bounded
neighborhood of the deleted rows. {The explicit local constructions for the residue classes of \(m\) and
\(n\) modulo \(5\) are given in Appendix~\ref{app:mod5-constructions}.} In each case, only the vertices in the neighborhood of the deleted rows are
affected, while the remaining part of the labeling stays periodic.
The local corrections guarantee that every zero-labeled
vertex still receives desired total weight from its closed neighborhood.
Thus the resulting labeling is a valid \([k]\)-Roman dominating function on
\(C_m\Box P_n\), and the stated upper bounds~follow.
\end{proof}}

{For clarity, Table~\ref{tab:mod5-summary} summarizes the five
residue classes appearing in Theorem~\ref{thm:mod5-final}.
The first term in each row corresponds to the construction obtained
from the next larger multiple of \(5\), while the subtracted terms
represent the savings obtained by deleting rows and applying local
corrections.

As a simple illustration, consider \(m=6\). Since \(6\equiv1\pmod5\),
Theorem~\ref{thm:mod5-final} gives
\[
\gamma_{[k]R}(C_6\Box P_n)
\le
2n(k+1)-1-2(k+2)\left\lfloor\frac{n-2}{5}\right\rfloor .
\]

The same residue class also contains the first exceptional
circumference not covered by the multiple-based constructions,
namely \(m=11\).
Since \(11\equiv1\pmod5\), Theorem~\ref{thm:mod5-final} yields
\[
\gamma_{[k]R}(C_{11}\Box P_n)
\le
3n(k+1)+2k-1
-2(k+2)\left\lfloor\frac{n-2}{5}\right\rfloor .
\]
Thus the modulo-\(5\) construction provides an explicit bound even
for circumferences such as \(11\), \(13\), and \(17\), which are not
multiples of any of the base values \(3,\dots,9\).
The optimality of these bounds remains open, but the example
illustrates the flexibility of the residue-class~method.}

\begin{table}[htbp]
\centering
\caption{Summary  of the residue-class upper bounds in
Theorem~\ref{thm:mod5-final}.}
\label{tab:mod5-summary}
{
\renewcommand{\arraystretch}{2.3}
\setlength{\tabcolsep}{7mm}
\begin{tabular}{cc}
\toprule
\textbf{Residue Class of \(m\)} & \textbf{Upper Bound for \(\gamma_{[k]R}(C_m\Box P_n)\)} \\
\midrule
\(m\equiv0\pmod5\) &
\(\displaystyle
\frac{m}{5}\bigl(n(k+1)+2k\bigr)
\) \\
\(m\equiv1\pmod5\) &
\(\displaystyle
\frac{m+4}{5}\bigl(n(k+1)+2k\bigr)
-(4k+1)-2(k+2)\left\lfloor\frac{n-2}{5}\right\rfloor
\) \\
\(m\equiv2\pmod5\) &
\(\displaystyle
\frac{m+3}{5}\bigl(n(k+1)+2k\bigr)
-3k-(k+3)\left\lfloor\frac{n-2}{5}\right\rfloor
\) \\
\(m\equiv3\pmod5\) &
\(\displaystyle
\frac{m+2}{5}\bigl(n(k+1)+2k\bigr)
-2k-2\left\lfloor\frac{n-2}{5}\right\rfloor
\) \\
\(m\equiv4\pmod5\) &
\(\displaystyle
\frac{m+1}{5}\bigl(n(k+1)+2k\bigr)-2k
\) \\
\bottomrule
\end{tabular}
}
\end{table}

\subsection{Uniform Constructions}

The uniform ceiling-type constructions obtained for the base cases $m=3,\ldots,9$
can likewise be extended to general cylindrical graphs.
By repeating the corresponding periodic pattern along the cycle direction, we obtain a unified upper bound of the following form.

\begin{thm}\label{thm:uniform_multiples_unified}
For $n\ge4$,
\[
\gamma_{[k]R}(C_m\Box P_n)
\le
m (n-2)\left\lceil\frac{k+4}{5}\right\rceil
+
2m\left\lceil
\frac{k+3-\left\lceil\frac{k+4}{5}\right\rceil}{3}
\right\rceil.
\]
\end{thm}

\begin{proof}
Let
\[
a=\left\lceil\frac{k+4}{5}\right\rceil
\qquad\text{and}\qquad
b=\left\lceil
\frac{k+3-a}{3}
\right\rceil.
\]
We assign weight $a$ to every vertex in each interior fibre
$F_i$, $1\le i\le n-2$.
On the two~boundary fibres $F_0$ and $F_{n-1}$, we assign weight $b$ to every vertex.

\noindent
{We first verify that this labeling is a $[k]$-RDF.
Let $v$ be an interior vertex.
If $v$ belongs to a fibre $F_i$ with $2\le i\le n-3$, then all vertices in $N[v]$
have weight $a$, and hence
\[
f(N[v])=5a\ge k+4.
\]

If $v$ belongs to $F_1$ or $F_{n-2}$, then one of its neighbors lies in a boundary fibre
and therefore has weight $b$, while the remaining vertices in $N[v]$ have weight $a$.
Thus
\[
f(N[v])\ge 4a+b.
\]
Since
\[
4a+b\ge k+4,
\]
every interior vertex satisfies the required condition.}

\noindent
For a boundary vertex $v$, its closed neighborhood contains the vertex itself,
its two neighbors on the cycle in the same fibre, and one neighbor in the adjacent interior fibre.
Hence
\[
f(N[v])\ge a+3b.
\]
By the definition of $b$, we have
\[
3b\ge k+3-a,
\]
and therefore
\[
f(N[v])\ge a+(k+3-a)=k+3.
\]
Since a boundary vertex has at most three neighbors, the $[k]$-Roman domination
condition also holds at every boundary vertex.

\noindent
Thus the constructed labeling is a valid $[k]$-RDF on $C_m\Box P_n$.
Its total weight is the sum of the contributions of the interior and boundary fibres.
There are $n-2$ interior fibres, each containing $m$ vertices of weight $a$,
so their total contribution~equals
\[
m(n-2)a.
\]
The two boundary fibres together contain $2m$ vertices of weight $b$,
so their total contribution~equals
\[
2mb.
\]
Consequently,
\[
\gamma_{[k]R}(C_m\Box P_n)
\le
m(n-2)\left\lceil\frac{k+4}{5}\right\rceil
+
2m\left\lceil
\frac{k+3-\left\lceil\frac{k+4}{5}\right\rceil}{3}
\right\rceil,
\]
as claimed.
\end{proof}

\subsection{Packing Constructions}

The packing-based refinements developed in the previous subsections can be extended to larger cylindrical graphs.
The key idea is that vertices forming a packing can have their weights reduced simultaneously without violating the $[k]$-Roman domination condition, which leads to an improvement of the basic ceiling-type upper bound.
This yields the following general estimate.

\begin{thm}\label{thm:packing_multiples_unified}
For $n\ge 4$,
\[
\gamma_{[k]R}(C_m\Box P_n)
\le
m (n-2)\left\lceil\frac{k+5}{5}\right\rceil
+
2m\left\lceil
\frac{k+{4}-\left\lceil\frac{k+5}{5}\right\rceil}{3}
\right\rceil
-
\rho(C_m\Box P_n),
\]
where $\rho(C_m\Box P_n)$ denotes the packing number of $C_m\Box P_n$.
\end{thm}

\begin{proof}
Let
\[
a'=\left\lceil\frac{k+5}{5}\right\rceil
\qquad\text{and}\qquad
b'=\left\lceil
\frac{k+{4}-a'}{3}
\right\rceil.
\]
We define a labeling $f$ by assigning weight $a'$ to every vertex of each interior fibre
$F_i$, $1\le i\le n-2$, and weight $b'$ to every vertex of the two boundary fibres.

\noindent
Let $v$ be an interior vertex. If $v$ belongs to a fibre $F_i$ with
$2\le i\le n-3$, then all vertices in $N[v]$ have weight $a'$, and hence
\[
f(N[v])=5a'\ge k+5.
\]
{If $v$ belongs to $F_1$ or $F_{n-2}$, then one of its neighbors lies in a boundary fibre
and has weight $b'$, while the remaining four vertices in $N[v]$ have weight $a'$.
Thus
\[
f(N[v])\ge 4a'+b'.
\]
Since
\[
4a'+b'\ge k+5,
\]
it follows that every interior vertex satisfies the required condition.

\noindent
On the other hand, for every boundary vertex $v$,
\[
f(N[v])\ge a'+3b'\ge k+{4}.
\]
Thus the neighborhood sums exceed the minimum requirement by at least one unit
at the relevant vertices, which creates the necessary slack.}

\noindent
{Let \(S\) be a maximum packing of \(C_m\Box P_n\).
By the definition of packing, the closed neighborhoods of distinct
vertices of \(S\) are pairwise disjoint.
The labeling constructed above has one extra unit of slack in every
closed neighborhood relevant to the \([k]\)-Roman domination condition.
Therefore, by Remark~\ref{rem:packing-reduction}, decreasing the
weight of every vertex of \(S\) by one preserves the
\([k]\)-Roman domination condition.

\noindent
The total weight of the resulting $[k]$-RDF is obtained from the
original ceiling-type construction by subtracting exactly one for
each vertex of the packing set \(S\).}
The original construction has total weight
\[
m(n-2)a' + 2mb',
\]
and the reduction contributes
\[
|S|=\rho(C_m\Box P_n).
\]
Consequently,
\[
\gamma_{[k]R}(C_m\Box P_n)
\le
m (n-2)\left\lceil\frac{k+5}{5}\right\rceil
+
2m\left\lceil
\frac{k+{4}-\left\lceil\frac{k+5}{5}\right\rceil}{3}
\right\rceil
-
\rho(C_m\Box P_n),
\]
which proves the theorem.
\end{proof}

\section{Upper bounds for the double Roman domination number of \texorpdfstring{$C_m\Box P_n$}{cylindrical grids}}\label{sec5}

For small values of $k$, numerical comparison of all derived bounds
shows that the linear constructions consistently provide the smallest
upper estimates among the three families (linear, uniform, and packing-based).
In particular, for $k=2$ the linear bounds dominate in all relevant
parameter ranges considered above.
Therefore, we state explicitly only the consequence of the linear
construction for the double Roman domination number.

For circumferences that are multiples of $3,\ldots,9$, the linear constructions
from Theorem~\ref{thm:linear_multiples_unified} yield the following bounds.

\begin{cor}\label{cor:double-roman-multiples}
Let $t\ge1$ and $n\ge2$. Then the following bounds hold:
\[
\gamma_{[2]R}(C_m\Box P_n)
\le
\begin{cases}
\displaystyle
t(2n+2),
& m=3t,\\[8pt]
\displaystyle
3tn,
& m=4t,\\[8pt]
\displaystyle
t(3n+4),
& m=5t,\\[10pt]
\displaystyle
t\Bigl(3\left\lceil \frac{4n}{3}\right\rceil+3\Bigr),
& m=6t,\ n\equiv 0 \pmod{3},\\[10pt]
\displaystyle
3t\left\lceil \frac{4n}{3}\right\rceil,
& m=6t,\ n\equiv 1 \pmod{3},\\[10pt]
\displaystyle
t\Bigl(3\left\lceil \frac{4n}{3}\right\rceil+2\Bigr),
& m=6t,\ n\equiv 2 \pmod{3},\\[12pt]
\displaystyle
t\frac{9n+11}{2},
& m=7t,\ n\text{ odd},\\[10pt]
\displaystyle
t\frac{9n+10}{2},
& m=7t,\ n\text{ even},\\[12pt]
\displaystyle
t\left(
6n
-
\left(
\left\lfloor\frac{n-2}{5}\right\rfloor
+
\left\lfloor\frac{n}{5}\right\rfloor
\right)
+4
\right),
& m=8t,\\[12pt]
\displaystyle
t(6n+4),
& m=9t.
\end{cases}
\]
\end{cor}

On the other hand, substituting $k=2$ into Theorem~\ref{thm:mod5-final}
yields a residue-class bound that applies to all cylindrical grids.

\begin{cor}\label{cor:double-roman-mod5}
Let $m\ge3$ and $n\ge4$. Then
\[
\gamma_{[2]R}(C_m\Box P_n)
\le
\begin{cases}
\dfrac{m}{5}(3n+4),
& m\equiv0\pmod5,\\[10pt]
\dfrac{m+4}{5}(3n+4)-9-8\left\lfloor\dfrac{n-2}{5}\right\rfloor,
& m\equiv1\pmod5,\\[10pt]
\dfrac{m+3}{5}(3n+4)-6-5\left\lfloor\dfrac{n-2}{5}\right\rfloor,
& m\equiv2\pmod5,\\[10pt]
\dfrac{m+2}{5}(3n+4)-4-2\left\lfloor\dfrac{n-2}{5}\right\rfloor,
& m\equiv3\pmod5,\\[10pt]
\dfrac{m+1}{5}(3n+4)-4,
& m\equiv4\pmod5.
\end{cases}
\]
\end{cor}

It is natural to compare Corollaries~\ref{cor:double-roman-multiples}
and~\ref{cor:double-roman-mod5}. The first one is based on constructions
tailored to special circumferences, whereas the second one applies to all
cylindrical grids. Hence the residue-class bound of
Corollary~\ref{cor:double-roman-mod5} has a wider range of applicability,
but it is not always numerically smaller for small values of $m$.

To determine which estimate is stronger, we compare the linear
coefficients with respect to $n$. For large $n$, the dominant term
determines which construction provides the better bound.

{\begin{thm}\label{thm:double-roman-asymptotic-comparison}
Let $m\ge3$ and $n\ge4$.

\begin{enumerate}
\item If $5\mid m$, then the bounds from
Corollaries~\ref{cor:double-roman-multiples} and
\ref{cor:double-roman-mod5} coincide.

\item For $5\nmid m$, $3\le m\le19$, the comparison is as follows:

\begin{itemize}
\item The bound from Corollary~\ref{cor:double-roman-multiples} is smaller for
\begin{itemize}
\item $m\in\{3,6,7,12\}$;
\item $m=8$ and $4\le n\le11$, or $n\ge12$ and $n\equiv0,1\pmod5$;
\item $m=16$ and $4\le n\le6$;
\item $m=18$ and $n\in\{4,\ldots,11,13,16\}.$
\end{itemize}

\item the bounds coincide for
\begin{itemize}
\item $m\in\{4,9\}$;
\item $m=8$ and $n\ge12$ and $n\not\equiv0,1\pmod5$;
\item $m=18$ and $n\in\{12,14,15,19\}.$
\end{itemize}

\item The bound from Corollary~\ref{cor:double-roman-mod5} is smaller for
\begin{itemize}
\item $m=14$;
\item $m=16$ and $n\ge7$;
\item $m=18$ and $n\in\{17,18\}$ or $n\ge20.$
\end{itemize}
\end{itemize}

\item For $5\nmid m$, $m\ge19$ and for sufficiently large $n$ the bound
from Corollary~\ref{cor:double-roman-mod5} is strictly smaller than every
applicable bound from Corollary~\ref{cor:double-roman-multiples}.
\end{enumerate}
\end{thm} }

\begin{proof}
We compare the dominant linear coefficients in $n$ appearing in
Corollaries~\ref{cor:double-roman-multiples} and
\ref{cor:double-roman-mod5}.

For Corollary~\ref{cor:double-roman-multiples}, admissible decompositions
$m=rt$ with $r\in\{3,\ldots,9\}$ yield asymptotic slopes reduced to the set
\[
\left\{
\frac{3m}{5},\;
\frac{9m}{14},\;
\frac{2m}{3},\;
\frac{3m}{4}
\right\}.
\]

For Corollary~\ref{cor:double-roman-mod5}
the asymptotic slopes are from the set
\[
\left\{
\frac{3m}{5},
\frac{3m+4}{5},
\frac{3m+3}{5}
\right\}.
\]

If $5\mid m$, both constructions yield the same bound. Therefore, assume that $5\nmid m$.
Among the slopes arising from Corollary~\ref{cor:double-roman-mod5},
the largest one equals
\[
\frac{3m+4}{5},
\]
while among the admissible slopes from
Corollary~\ref{cor:double-roman-multiples} the smallest one equals
\[
\frac{9m}{14}.
\]
A direct comparison shows that
\[
\frac{3m+4}{5}<\frac{9m}{14}
\]
if and only if
\[
42m+56<45m,
\]
that is, if and only if $m>\frac{56}{3}$.
Hence for every $m\ge19$ with $5\nmid m$, the residue-class construction
has strictly smaller slope than any construction from
Corollary~\ref{cor:double-roman-multiples}. Consequently,
Corollary~\ref{cor:double-roman-mod5} yields the better bound for all
sufficiently large $n$ in this range.

It remains to compare the finitely many cases $3\le m\le18$ directly.
For each admissible value of $m$ in this range, we evaluated the explicit
expressions from Corollaries~\ref{cor:double-roman-multiples} and
\ref{cor:double-roman-mod5} and compared the resulting bounds as
functions of $n$ for all $n\ge4$.

More precisely, for every fixed $m$ we first determined all applicable
decompositions $m=rt$ with $r\in\{3,\dots,9\}$ and selected the smallest
corresponding bound from Corollary~\ref{cor:double-roman-multiples}.
We then compared this bound with the residue-class bound from
Corollary~\ref{cor:double-roman-mod5}.
Since both expressions are linear functions of $n$ up to periodic floor
terms, their difference can be analyzed explicitly. In particular, by
comparing the linear coefficients in $n$ we determined the asymptotic
ordering of the bounds, and for the remaining cases we computed the exact
values of the difference for $n\ge4$ in order to identify the transition
points where the ordering changes.

We illustrate the above comparison procedure in the case $m=16$.
The admissible decompositions
$16=4\cdot4=8\cdot2$ yield from
Corollary~\ref{cor:double-roman-multiples} the bound
\[
\min\!\left\{
12n,\;
12n
-
2\left(
\left\lfloor\frac{n-2}{5}\right\rfloor
+
\left\lfloor\frac{n}{5}\right\rfloor
\right)
+8
\right\}.
\]
{Since $16\equiv1\pmod5$, Corollary~\ref{cor:double-roman-mod5} gives
\[
12n+7-8\left\lfloor\frac{n-2}{5}\right\rfloor.
\]
Comparing the two expressions for $n\ge4$ shows that the
multiple-based bound is smaller for $4\le n\le6$, while the residue-class
bound is smaller for $n\ge7$.}
\end{proof}

Theorem~\ref{thm:double-roman-asymptotic-comparison} shows that the
multiple-based constructions remain optimal only for a small number of
exceptional circumferences, while the
residue-class construction provides the strongest general upper bounds
for all sufficiently large $m$ and $n$. Together, the two~approaches therefore
yield effective estimates for $\gamma_{[2]R}(C_m\Box P_n)$ for all
cylindrical grids.

\section{Conclusions}\label{sec6}

In this paper, we studied $[k]$-Roman domination on cylindrical grids
$C_m\Box P_n$ and developed several constructive upper bounds based on
three different approaches: linear periodic constructions, uniform
ceiling-type assignments, and packing-based refinements.
For the special case $C_9\Box P_n$, we compared these three families in
detail and showed that their relative performance depends essentially on
the parameter $k$.
For small values of $k$, the linear constructions provide the best
bounds, whereas for larger values of $k$ the uniform and packing-based
constructions become asymptotically superior.

We then extended the linear constructions to all cylindrical grids whose
circumference is a multiple of one of the values $3,\dots,9$, obtaining
a unified family of explicit upper bounds.
Motivated by the fact that blocks of length five give the smallest
asymptotic slope for $k\ge2$, we further derived a general residue-class
construction depending on $m\bmod 5$.
This yielded upper bounds for $\gamma_{[k]R}(C_m\Box P_n)$ for all
cylindrical grids and, in particular, explicit consequences for the
double Roman domination number.

For $k=2$, we showed that the linear constructions provide the strongest
bounds among the three general families considered here.
This allowed us to compare in detail the multiple-based bounds and the
modulo-$5$ construction.
The comparison shows that the residue-class construction is asymptotically
best for all sufficiently large admissible circumferences, while the
multiple-based constructions remain optimal only in several exceptional
small cases.

{The obtained upper bounds should be viewed in relation to the
known density-type lower estimates for cylindrical grids.
In earlier work, lower bounds based on local neighborhood constraints
and efficient domination show that the natural benchmark for
\(\gamma_{[k]R}(C_m\Box P_n)\) is of the order
\[
\frac{(k+1)mn}{5}.
\]
Thus the constructions developed in the present paper have the correct
linear dependence on both \(m\) and \(n\), while the remaining gap is
mainly reflected in the precise leading constants and in additive
boundary terms.
Closing this gap remains one of the main challenges for future work.}

{In this way, the present paper completes the analysis of the
previously untreated base case $C_9\Box P_n$ and closes the
remaining gap in the sequence of periodic constructions for
cylindrical grids with small fixed circumferences.
Combined with the earlier results for
$m\in\{3,4,\dots,8\}$, the constructions developed here now
provide a unified framework for constructive upper bounds on
$[k]$-Roman domination for all cylindrical grids.

The obtained results also refine several previously known bounds
for double Roman domination on Cartesian product graphs.
In particular, Corollary~\ref{cor:double-roman-multiples}
improves Theorem~3.3 of Anu and Lakshmanan~\cite{AnuAparna2024}
for cylindrical graph families of the form
$C_3\Box P_n \cong K_3\Box P_n$.

More generally, the present paper extends such product-based
results from the special case \(k=2\) to the full
\([k]\)-Roman domination framework and shows that constructions
adapted to cylindrical grids lead to sharper estimates than those
obtained from general Cartesian product techniques alone.}

There are several natural directions for future work.
{First, the general bounds in Theorem~\ref{thm:mod5-final} are formulated
only in terms of the residue class of $m$ modulo $5$.
However, the terminal corrections also depend on the residue class of
$n$ modulo $5$, because the periodic pattern is truncated differently at
the two ends of the path.
The constructions collected in Appendix~\ref{app:mod5-constructions} illustrate this dependence.

For example, in the case $m\equiv1\pmod5$ and $n\equiv3\pmod5$, the
local correction allows a saving of $2k+1$ at both ends of the path.
In the uniform statement of Theorem~\ref{thm:mod5-final}, however, we
use a correction that is valid for all values of $n$: one obtains a
saving of $2k+1$ at one end and \(2k\) at the other end.
Thus a separate treatment of the five residue classes of $n$ modulo
$5$ could slightly improve the additive constants in the main theorem.
Such refinements would not change the asymptotic slope in $n$, but they
could sharpen the bounds for fixed and moderate values of $n$.}

{Second, the comparison carried out for the double Roman domination number
suggests that several small circumferences still deserve special
attention.
In particular, the graphs $C_{11}\Box P_n$, $C_{13}\Box P_n$, and
$C_{17}\Box P_n$ are not multiples of any of the base circumferences
$3,\dots,9$ considered in the multiple-based construction.
Although the modulo-$5$ construction provides explicit upper bounds for
these graphs, it is not clear whether these bounds are optimal, or even
asymptotically best among constructions adapted to these particular~circumferences.

A possible way to treat these missing cases would be to develop
tailored constructions for $C_{11}\Box P_n$, $C_{13}\Box P_n$, and
$C_{17}\Box P_n$, and to combine them with packing-based refinements.
This would require determining the
corresponding packing numbers
$\rho(C_{11}\Box P_n)$, $\rho(C_{13}\Box P_n)$, and $\rho(C_{17}\Box P_n).$
Then one could compare the
resulting linear coefficients in $n$ with those of the general
modulo-$5$ construction, as was done above for the double Roman case.
This remains an interesting open problem.}

Finally, another challenging direction is to determine exact values, or
at least asymptotically tight formulas, for $\gamma_{[k]R}(C_m\Box P_n)$
for fixed small values of $k$.
The periodic structure of cylindrical grids suggests that sharper
characterizations may be possible, especially for double and triple Roman
domination.
{A possible approach to obtain exact formulae  is to use the algebraic method as in~\cite{Kragujevac,MartinezGarzonPuertas}.
While the method is universal it  is  computationaly intensive, and consequently,
 only  some partial results for relative small $m$ are obtained until now.
}

\section*{Funding}

The  first author (S.B.) acknowledges financial support from the Slovenian Research and Innovation Agency (ARIS)
through research programme No.\ P1-0297 and research project No.\ J1-70016.
The second author (J.Z.) was partially supported by ARIS through the annual work program of Rudolfovo
and by the research grants P2-0248 and L1-60136.
This work was also supported in part by the Horizon Europe project Quantum Excellence Centre for Quantum-Enhanced Applications~(QEC4QEA).

\section*{Data availability}
During the research no new data were created.

{\section*{Acknowledgments}
The authors wish to sincerely thank to the four anonymous reviewers for their constructive remarks that helped us to considerably improve the manuscript. }
\section*{Conflicts of interest}
The {authors declare} no conflicts of interest.

\clearpage
\appendix
\section{Local Constructions for the Modulo-{5} Residue Classes}
\label{app:mod5-constructions}

{This appendix contains the local constructions used in the proof of
Theorem~\ref{thm:mod5-final}.
They describe how the periodic construction based on blocks of length
five can be modified when the circumference \(m\) is not divisible by
\(5\).
The four cases correspond to the residue classes
\(m\equiv4,3,2,1\pmod5\), respectively.
Within each case, the subcases depend on the residue class of \(n\)
modulo \(5\).
The gray marks indicate the local row deletions and boundary corrections
needed to preserve the \([k]\)-Roman domination condition.}

\begin{figure}[H]
\centering
\includegraphics[page=1,width=\textwidth,height=0.70\textheight,keepaspectratio]{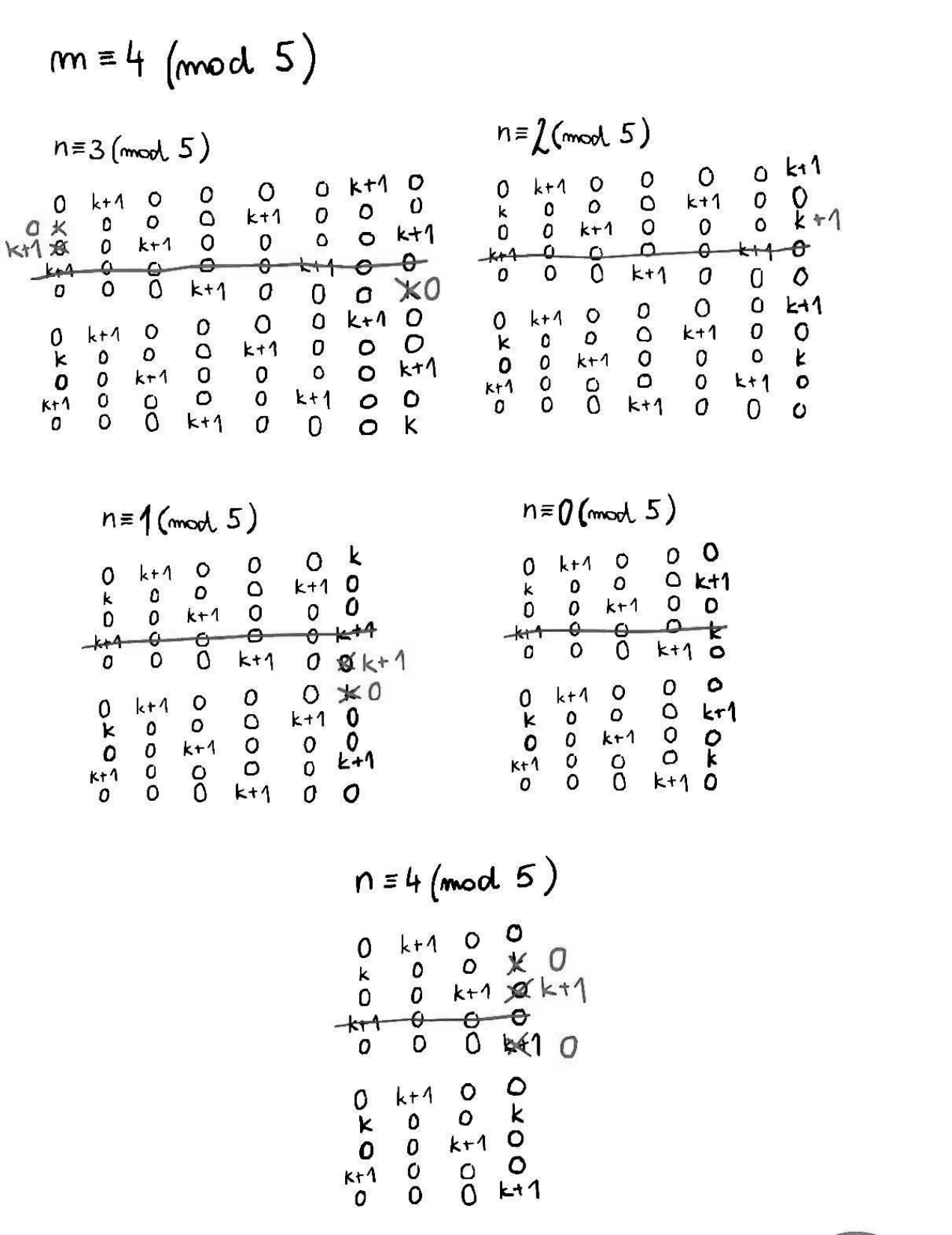}
\caption{Local constructions for the case \(m\equiv4\pmod5\).}
\label{fig:app-m4}
\end{figure}

\clearpage
\begin{figure}[p]

\centering
\includegraphics[page=2,width=\textwidth,height=0.90\textheight,keepaspectratio]{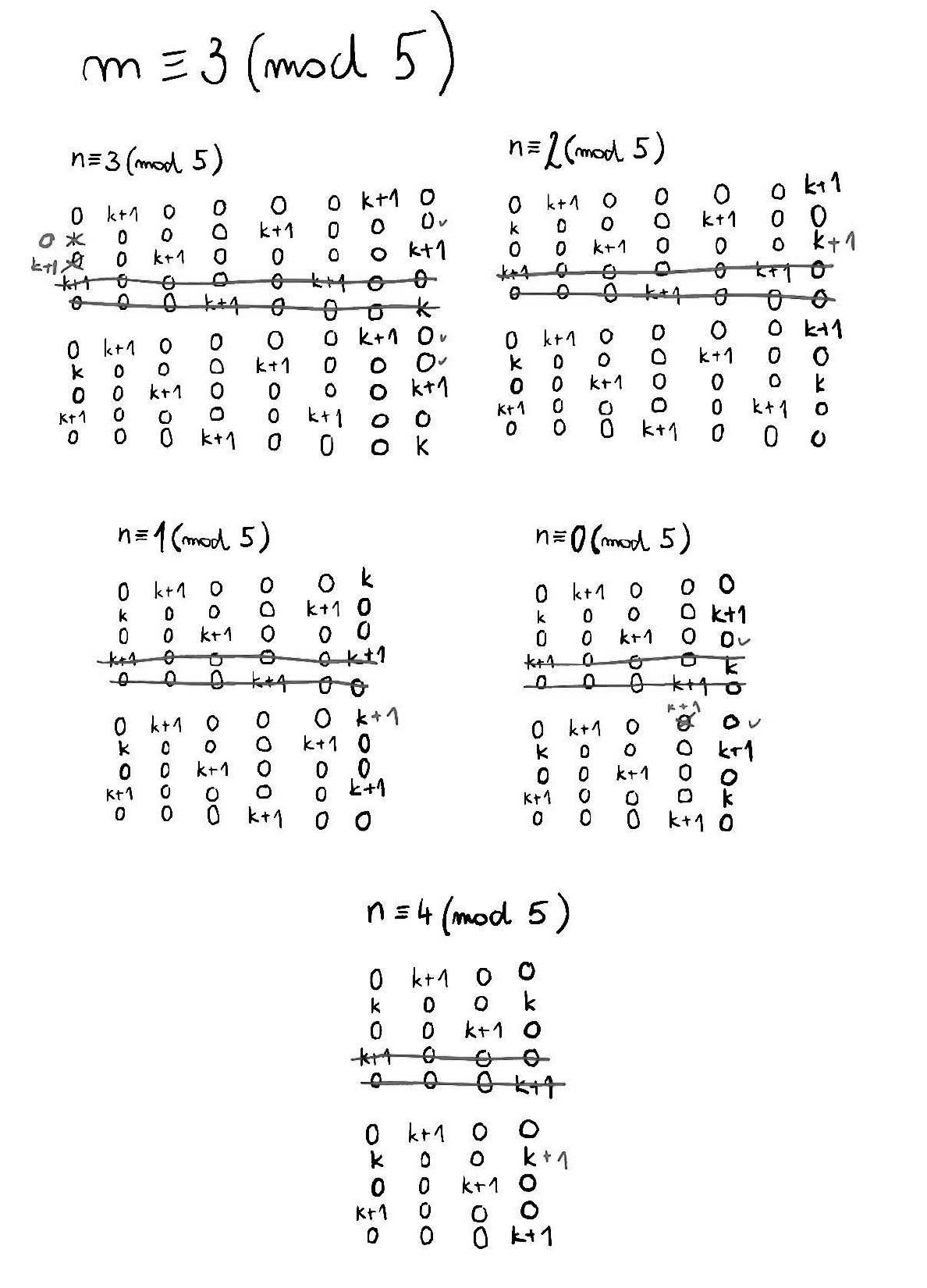}

\caption{Local constructions for the case \(m\equiv3\pmod5\).}
\label{fig:app-m3}
\end{figure}

\clearpage
\begin{figure}[p]

\centering
\includegraphics[page=3,width=\textwidth,height=0.90\textheight,keepaspectratio]{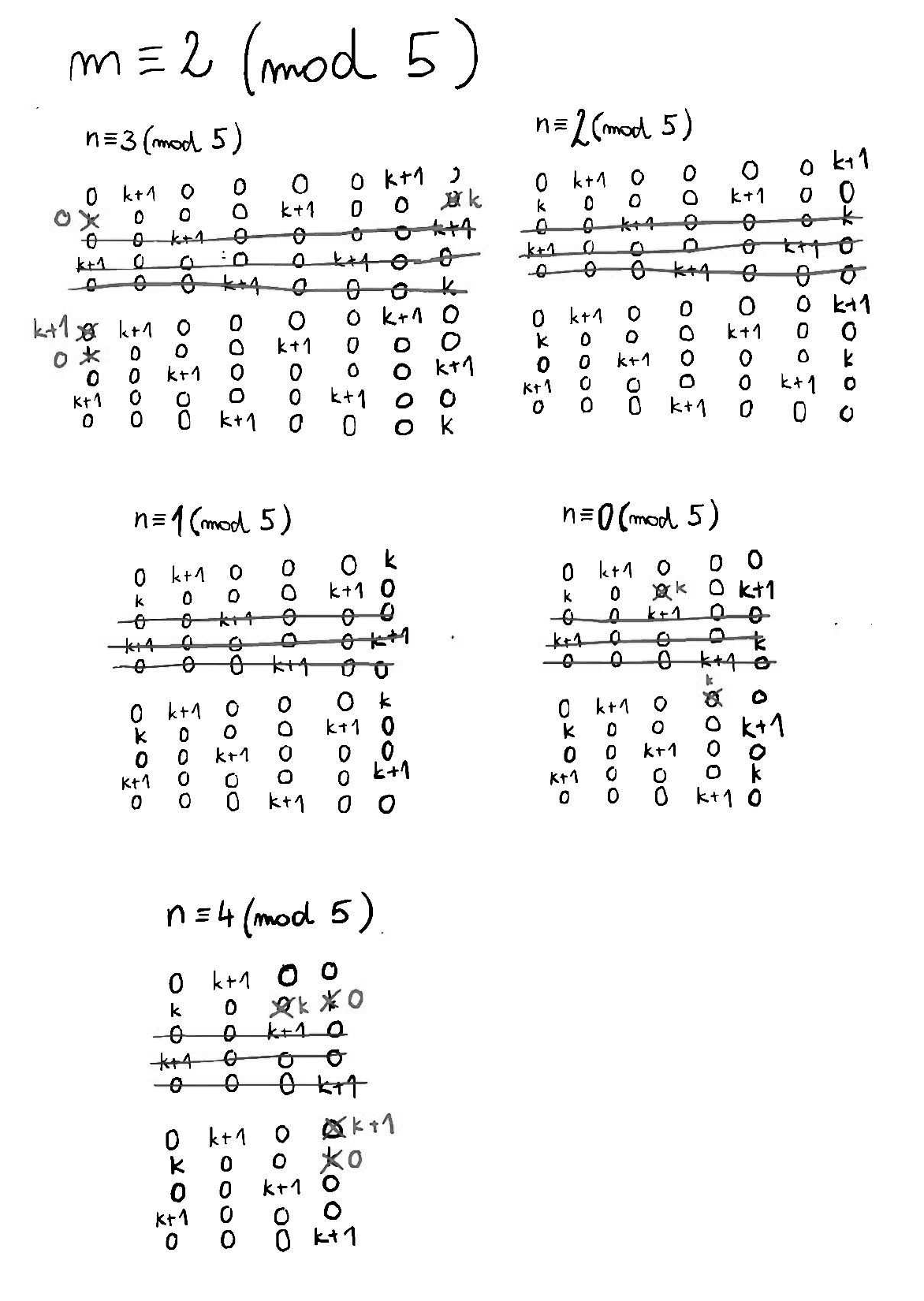}

\caption{Local constructions for the case \(m\equiv2\pmod5\).}
\label{fig:app-m2}
\end{figure}

\clearpage
\begin{figure}[p]

\centering
\includegraphics[page=4,width=\textwidth,height=0.90\textheight,keepaspectratio]{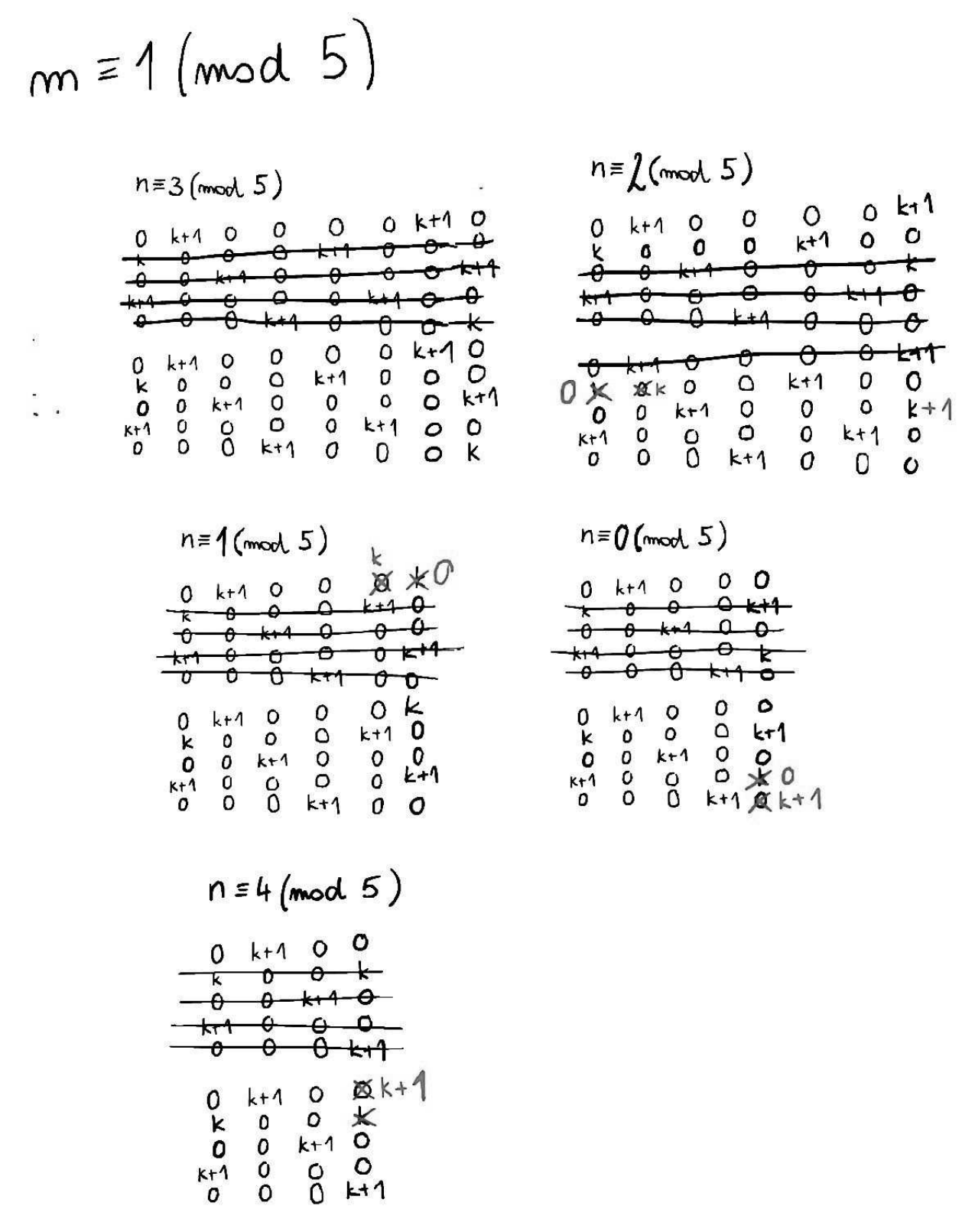}

\caption{Local constructions for the case \(m\equiv1\pmod5\).}
\label{fig:app-m1}
\end{figure}


\clearpage
\begin{thebibliography}{20}

\bibitem{Haynes}
Haynes, T.; Hedetniemi, S.; Slater, P.
\newblock {\em Fundamentals of Domination in Graphs}, 1st ed.; CRC Press:  Boca Raton, FL, USA,   1998.
\newblock {\url{https://doi.org/10.1201/9781482246582}}.

\bibitem{GareyJohnson}
Garey, M.R.; Johnson, D.S.
\newblock {\em Computers and Intractability: A Guide to the Theory of
  NP-Completeness}; W. H. Freeman \& Co.: New York, NY, USA,  1979.

\bibitem{Gan2025GraphCSR}
Gan, X.; Li, T.; Gong, C.; Li, D.; Dong, D.; Liu, J.; Lu, K.
\newblock GraphCSR: A Degree-Equalized CSR Format for Large-scale Graph
  Processing.
\newblock {\em Proc. VLDB Endow.} {\bf 2025}, {\em
  18},~4255--4268.
\newblock {\url{https://doi.org/10.14778/3749646.3749691}}.

\bibitem{Khemani2024GNNReview}
Khemani, B.; Patil, S.; Kotecha, K.; Tanwar, S.
\newblock A review of graph neural networks: concepts, architectures,
  techniques, challenges, datasets, applications, and future directions.
\newblock {\em J. Big Data} {\bf 2024}, {\em 11},~18.
\newblock {\url{https://doi.org/10.1186/s40537-023-00876-4}}.

\bibitem{Cockayne2004}
Cockayne, E.J.; Dreyer, P.A.; Hedetniemi, S.M.; Hedetniemi, S.T.
\newblock {{{Roman}}} domination in graphs.
\newblock {\em Discret. Math.} {\bf 2004}, {\em 278},~11--22.
\newblock {\url{https://doi.org/10.1016/j.disc.2003.06.002}}.

\bibitem{Beeler2016}
Beeler, R.A.; Haynes, T.W.; Hedetniemi, S.T.
\newblock Double {{{Roman}}} domination.
\newblock {\em Discret. Appl. Math.} {\bf 2016}, {\em 211},~23--29.
\newblock {\url{https://doi.org/10.1016/j.dam.2016.03.017}}.

\bibitem{D.Survey}
Rupnik~Poklukar, D.; \v{Z}erovnik, J.
\newblock Double {{Roman}} Domination: A Survey.
\newblock {\em Mathematics} {\bf 2023}, {\em 11}, 351.
\newblock {\url{https://doi.org/10.3390/math11020351}}.

\bibitem{ABDOLLAHZADEHAHANGAR2021125444}
{Abdollahzadeh Ahangar}, H.; Álvarez, M.; Chellali, M.; Sheikholeslami, S.;
  Valenzuela-Tripodoro, J.
\newblock Triple {{{Roman}}} domination in graphs.
\newblock {\em Appl. Math. Comput.} {\bf 2021}, {\em
  391},~125444.
\newblock {\url{https://doi.org/10.1016/j.amc.2020.125444}}.

\bibitem{Chellali2021VarietiesI}
Chellali, M.; Jafari~Rad, N.; Sheikholeslami, S.M.; Volkmann, L.
\newblock Varieties of {{{{Roman}}}} Domination. In {\em Structures of
  Domination in Graphs}; {Developments in Mathematics}; Haynes, T.W., Hedetniemi, S.T., Henning, M.A., Eds.;
  Springer: Cham, Switzerland, 2021; Volume~66,  pp.
  273--307.
\newblock {\url{https://doi.org/10.1007/978-3-030-58892-2_10}}.

\bibitem{Chellali2020VarietiesII}
Chellali, M.; Jafari~Rad, N.; Sheikholeslami, S.M.; Volkmann, L.
\newblock Varieties of {{{{Roman}}}} domination {II}.
\newblock {\em AKCE Int. J. Graphs Comb.} {\bf
  2020}, {\em 17},~966--984.
\newblock {\url{https://doi.org/10.1016/j.akcej.2019.12.001}}.

\bibitem{ImrichKlavzar}
Imrich, W.; Klavzar, S.
\newblock {\em Product Graphs: Structure and Recognition}; Wiley-Interscience:
  New York, NY, USA, 2000.

\bibitem{Chang1992}
Chang, T.Y.
\newblock Domination Numbers of Grid Graphs.
\newblock {Ph.D. Thesis, Department of Mathematics, University of South Florida},  Tampa, FL, USA, 
  {1992}.

\bibitem{Alanko}
Alanko, S.; Crevals, S.; Isopoussu, A.; Ostergard, P.; Pettersson, V.
\newblock Computing the domination number of grid graphs.
\newblock {\em  Electron. J. Comb.} {\bf 2011}, {\em 18}, 141.

\bibitem{Ahangar2021}
Valenzuela-Tripodoro, J.C.; Mateos-Camacho, M.A.; Cera, M.;  Alvarez-Ruiz, M.P. \newblock The ${[}k{]}$-multiple {{{Roman}}} domination in graphs.
\newblock {\em Res. Sq.} {\bf 2023}, \textit{preprint}.
  {\url{https://doi.org/10.21203/rs.3.rs-2889100/v1}}.

\bibitem{Khalili2023}
Khalili, N.; Amjadi, J.; Chellali, M.; Sheikholeslami, S.M.
\newblock On {[}k{]}-{{{Roman}}} domination in graphs.
\newblock {\em AKCE Int. J. Graphs Comb.} {\bf
  2023}, {\em 20},~291--299.
\newblock {\url{https://doi.org/10.1080/09728600.2023.2241531}}.

\bibitem{AnuAparna2024}
Anu, V.; Lakshmanan, S.A.
\newblock Double {Roman} Domination in {Cartesian} Product.
\newblock {\em Creat. Math. Inform.} {\bf 2024}, {\em 33}, 1--6.
\newblock {\url{https://doi.org/10.37193/CMI.2024.01.01}}.

\bibitem{BrezovnikZerovnik2026_1}
Brezovnik, S.; {\v{Z}}erovnik, J.
\newblock {$[k]$}-{{Roman}} Domination on Cylindrical Grids {$C_m \square P_n$}. \textit{arXiv} {\bf 2026}, arXiv:2603.02831.
  
\bibitem{BrezovnikZerovnikkRomanCylindrical2}
Brezovnik, S.; {\v{Z}}erovnik, J.
\newblock Further results on {$[k]$}-{Roman} domination on cylindrical grids
  {$C_m \square P_n$}. \textit{arXiv} {\bf 2026}, arXiv:2603.25191.
  
\bibitem{Kragujevac}
Pavli\v{c}, P.; \v{Z}erovnik, J.
\newblock A note on the domination number of the {C}artesian products of paths
  and cycles.
\newblock {\em Kragujev. J. Math.} {\bf 2013}, {\em
  37},~275--285.

\bibitem{MartinezGarzonPuertas}
Martínez, J.A.; Garzón, E.M.; Puertas, M.L.
\newblock Powers of Large Matrices on GPU Platforms to Compute the Roman
  Domination Number of Cylindrical Graphs.
\newblock {\em IEEE Access} {\bf 2021}, {\em 9},~29346--29355.
\newblock {\url{https://doi.org/10.1109/ACCESS.2021.3058738}}.

\end{thebibliography}
\end{document}